\documentclass[11pt]{article}
\usepackage{latexsym}
\usepackage{amsmath,amsthm,amsfonts,amssymb,graphicx,epsfig,latexsym,color}
\usepackage{epsf,enumerate}
\usepackage{pictex,dcpic}

\def\int{\operatorname{int}}

\newtheorem{thm}{Theorem}[section]
\newtheorem{lemma}[thm]{Lemma}
\newtheorem{cor}[thm]{Corollary}
\newtheorem{prop}[thm]{Proposition}
\newtheorem*{main}{Main Theorem}{}

\theoremstyle{remark}
\newtheorem{example}[thm]{Example}

\newtheorem{definition}[thm]{Definition}
\newtheorem{remark}[thm]{Remark}

\newtheorem*{definition*}{Definition}
\newtheorem*{remark*}{Remark}

\newtheorem{notation}[thm]{Notation}

\def\R{{\mathbb R}}

\def\Z{{\mathbb Z}}

\def\S{{\mathcal S}}

\def\F{{\mathcal F}}
\def\X{{\mathcal X}}

\def\diam{\operatorname{diam}}
\def\FF{{\mathbb F}}

\def\XS{\mathcal C} 

\newcommand{\<}{\langle}
\renewcommand{\>}{\rangle}

\title{Subfactor projections}

\author{Mladen Bestvina and Mark Feighn\thanks{Both
    authors gratefully acknowledge the support by the National
    Science Foundation.}}
\date{January 26, 2013}

\begin{document}

\maketitle

\begin{abstract}{When two free factors $A,B$ of a free group $\FF_n$ are in
  ``general position'' we define the projection $\pi_A(B)$ of $B$ to the
  splitting complex (alternatively, the complex of free factors)
  $\S(A)$ of $A$. We show that the projections satisfy properties
  analogous to subsurface projections introduced by Masur and
  Minsky. We use the subfactor projections to construct an action of
  $Out(\FF_n)$ on a finite product of hyperbolic spaces where every
  automorphism with exponential growth acts with positive translation
  length. We also prove a version of the Bounded geodesic image
  theorem.
In the appendix, we give a sketch of the proof of the Handel-Mosher
hyperbolicity theorem for the splitting complex using (liberal)
folding paths.}
\end{abstract}

\section{Introduction}

The work of Howard Masur and Yair Minsky \cite{MM,MM2} has served as
an inspiration in the study of $Out(\FF_n)$, the outer
automorphism group of the free group $\FF_n$ of rank $n$. Their work
starts by proving that the curve complex associated to a compact
surface is hyperbolic. There are several established analogs of the curve
complex in the $Out(\FF_n)$ setting that have recently
been shown to be hyperbolic, see \cite{F,S,Brian} and \cite[Section 8]{HH}.

The complex $\F$ of free factors is the simplicial complex whose
vertices are conjugacy classes of nontrivial free factors of $\FF_n$ and
a simplex is formed by those vertices that have representatives that
can be arranged in a chain. (A subgroup $A$ of $\FF_n$ is a nontrivial
free factor if $\FF_n=A*B$ for another subgroup $B$ with $A\neq 1\neq
B$.) 

The complex $\S$ of free splittings, or the {\it sphere complex}, or
the simplicial completion of Outer space, is the simplicial complex
whose $k$-simplices are $(k+1)$-edge free splittings (nontrivial
minimal reduced simplicial $\FF_n$-trees with trivial edge stabilizers
and $(k+1)$ orbits of edges), up to
equivariant isomorphism.

\begin{thm}[\cite{F}]\label{bf}
$\F$ is hyperbolic.
\end{thm}

\begin{thm}[Handel-Mosher \cite{S}]\label{hm2}
$\S$ is hyperbolic.
\end{thm}

Kapovich-Rafi \cite{kapovich-rafi} show how to derive Theorem \ref{bf}
from Theorem~\ref{hm2} in a simple way, and Hilion-Horbez \cite{HH} give a
simpler proof of Theorem~\ref{hm2} using surgery paths.

Masur and Minsky also define {\it subsurface projections}. If $\alpha$
is a simple closed curve in a surface $S$ and $A$ is a subsurface such
that $\alpha$ and $A$ cannot be made disjoint by an isotopy (we also
say they {\it intersect}), there is a curve $\alpha|A$ in the curve
complex of $A$, well defined up to a bounded set, obtained by
``closing up the intersection of $\alpha$ and $A$ along the boundary
of $A$''. If $A,B$ are two subsurfaces that intersect, one defines the
projection $\pi_A(B)$ of $B$ to $A$ by taking $\partial B|A$. We view
$\pi_A(B)$ as a bounded subset of the curve complex of $A$.  Masur and
Minsky use subsurface projections to estimate distances in mapping
class groups, and more generally, to understand their large-scale
geometry. To encapsulate this, we first recall a theorem.

\begin{thm}[\cite{BBF}]\label{bbf}
Let $\mathcal Y$ be a collection of $\delta$-hyperbolic spaces and
for every pair $A,B\in\mathcal Y$ of distinct elements suppose we are
given a uniformly bounded subset $\pi_A(B)\subset A$, called the {\emph
  {projection}} of $B$ to $A$. Denoting by $d_A(B,C)$ the diameter of
$\pi_A(B)\cup \pi_A(C)$, assume the following holds: there is a
constant $K>0$ such that
\begin{itemize}
\item if $A,B,C\in\mathcal Y$ are distinct, at most one of 3 numbers
  $$d_A(B,C), d_B(A,C), d_C(A,B)$$ is $>K$, and
\item for any distinct $A,B$ the set
$$\{C\in\mathcal Y-\{A,B\}\mid d_C(A,B)>K\}$$ is finite.
\end{itemize}
Then there is a hyperbolic space $Y$ and an isometric embedding of
each $A\in\mathcal Y$ onto a convex set in $Y$ so the images are
pairwise disjoint and the nearest point projection of any $B$ to any
$A\neq B$ is within uniformly bounded distance of
$\pi_A(B)$. Moreover, the construction is equivariant with respect to
any group acting on $\mathcal Y$ by isometries.
\end{thm}

In the setting of subsurface projections, $\mathcal Y$ will be the
collection of curve complexes associated to a collection of pairwise
intersecting subsurfaces.
That both bullets hold in the setting of subsurface projections
(whenever they are both defined) is the work of Masur-Minsky \cite{MM,MM2} and
Behrstock \cite{Behrstock}. There is a simple proof of the first
bullet due to Leininger, see \cite{johanna}, and of the second bullet
in \cite{BBF}. 

It is further shown in \cite{BBF} that there is a finite equivariant
coloring of the collection of all subsurfaces of $S$, and the
boundaries of two distinct subsurfaces with the same color intersect.
Applying Theorem \ref{bbf} to each color separately and then taking
the cartesian product one obtains:

\begin{thm}[\cite{BBF}]
The mapping class group of $S$ acts by isometries on the finite product
$Y_1\times\cdots\times Y_k$ of hyperbolic spaces. Moreover, orbit maps
are quasi-isometric embeddings.
\end{thm}

The second statement in the theorem recasts the Masur-Minsky distance
estimate.

The goal of this paper is to take a step in this direction in the
$Out(\FF_n)$ setting. If $A$ and $B$ are distinct conjugacy
classes of nontrivial free factors, we attempt to define $\pi_A(B)$, a
bounded subset of the splitting complex $\S(A)$ of $A$. We start with
a simplicial $\FF_n$-tree $T$ representing a splitting of
$\FF_n$ where $B$ fixes a vertex. Restricting the action to
$A$ and passing to the minimal subtree (assuming $A$ does not fix a
vertex of $T$) gives a splitting of $A$. Different choices of $T$ will
in general produce unbounded sets in $\S(A)$, but when $A$ and $B$ are
in ``general position'' the ambiguity is bounded.

\begin{main}
Assume $n=rank(\FF_n)\geq 3$.
The set of vertices of $\F$ can be equivariantly colored in finitely
many colors so that when $A$ and $B$ are distinct and either have the
same color or are at distance $>4$ then the projection $\pi_A(B)$
is well defined as a bounded subset of $\S(A)$. Moreover, restricting
to one color, both bullets in Theorem \ref{bbf} hold.
\end{main}

The Main Theorem is proved in Section 4.
We thus obtain an isometric action of $Out(\FF_n)$ on a finite product
of hyperbolic spaces. However, orbit maps are not quasi-isometric
embedding. We discuss this further in Section \ref{further}.
\vskip 0.3cm

\noindent
{\bf Theorem 5.1.} {\it $Out(\FF_n)$ acts on a finite product
  $Y=Y_1\times\cdots\times Y_k$ of hyperbolic spaces so that every
  exponentially growing automorphism has positive translation length.}

\vskip 0.3cm

The key to understanding the projections is the analysis of a path in
Outer space $\X(A)$ of $A$ induced by a path $G_t$ in $\X$. Here is a
motivating example. For a review of the terminology, see Section
\ref{review}. 

Consider the automorphism $\Phi$ of ${\mathbb F}_3=\langle
x,y,z\rangle$ given by $\phi:{\mathbb F}_3\to {\mathbb F}_3$,
$x\mapsto y$, $y\mapsto z$ and $z \mapsto zx$. The map on the rose
$R_0$ with loops representing $x,y,z$ is a train track map for
$\Phi$. The vertex has 4 gates, namely $\{x,y,z\}, \{X\}, \{Y\}$ and
$\{Z\}$ (the direction oriented the same way as the edge $x$ is also
called $x$; the direction pointing the opposite way is $X$ etc). The
dilatation $\lambda$ of $\Phi$ satisfies $\lambda^3-\lambda^2-1=0$ and
$\lambda= 1.46557\cdots$.  The lengths of edges $x,y,z$ are
proportional to $1,\lambda,\lambda^2$.

Let $A=\langle x,y\rangle$. We are interested in picturing the core of
the covering space corresponding to $A$ of the graphs $R_k=R_0\Phi^k$,
which we view as a discrete path in $\X$ as $k\in\Z$ (but one can
extend it to an axis $R_k$, $k\in\R$ of $\Phi$). This core has an
induced metric and the structure of legal and illegal turns.  Since
$\ell_{R_0\Phi^k}(\gamma)=\ell_{R_0}(\Phi^k(\gamma))$, this amounts to
picturing the $\phi^k(A)$-cover of $R_0$. See Figure \ref{example}.

\begin{figure*}[h]
\centering
\includegraphics[scale=0.15]{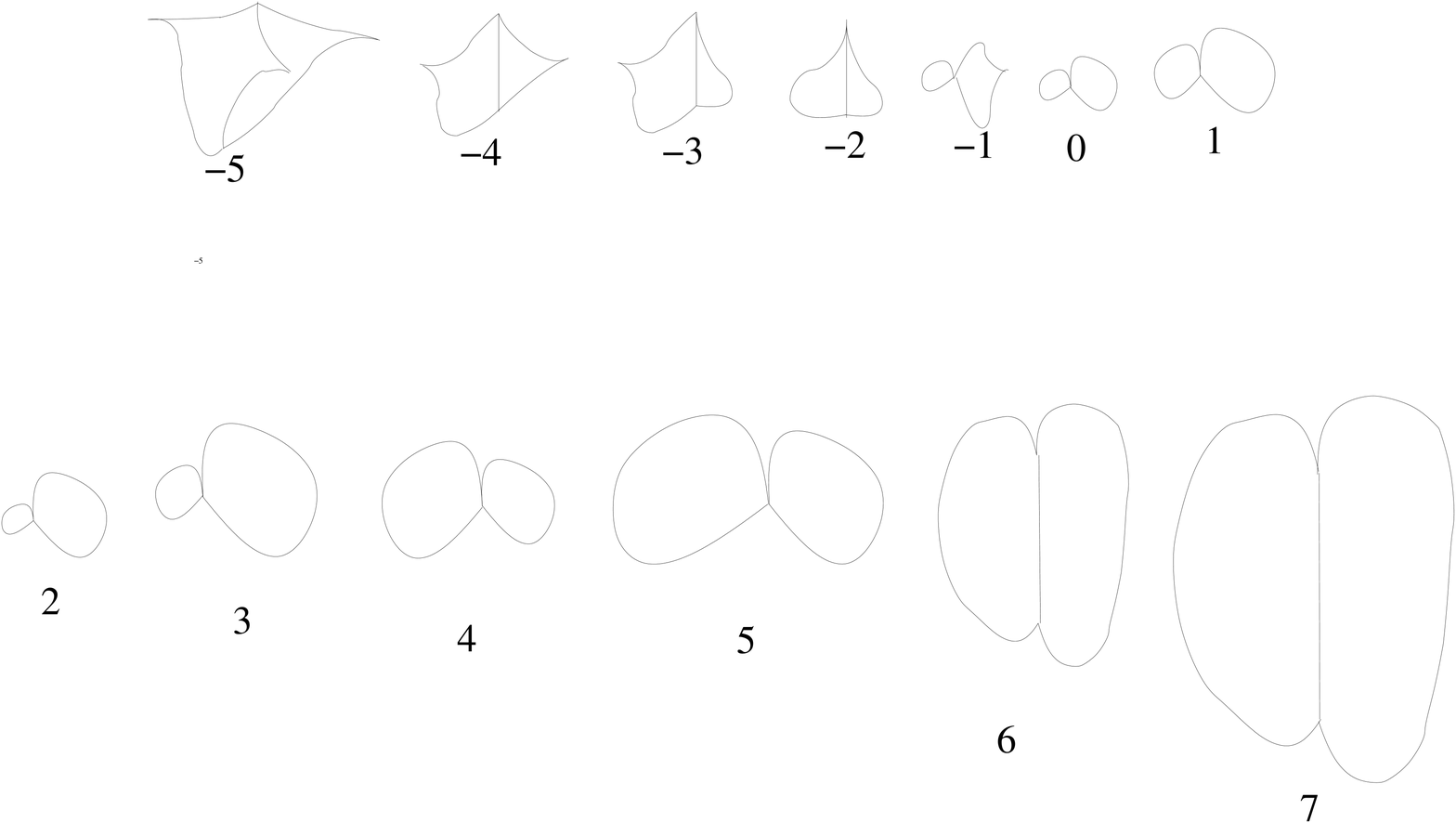}
\caption{$A|R_0\Phi^k$ for $-5\leq k\leq 7$.}
\label{example}
\end{figure*}

For example, for $k=2$ we compute $\phi^2(A)=\langle
z,xz\rangle=\langle x,z\rangle$, and for $k=6$ we get
$\phi^6(A)=\langle zxyzzx,zxyzzxzxy\rangle=\langle zxy,zzx\rangle$
which is represented by a theta graph with the middle edge $z$ going
up, the left edge $xy$ and the right edge $zx$, both going down. The
illegal turn is at the top.

The reader should note the following features:
\begin{itemize}
\item As $k\to \infty$ the volume of $A|R_k$ goes to $\infty$ and in
  fact the topological edges of $A|R_k$ are legal and grow
  exponentially.
\item As $k\to -\infty$ the volume of $A|R_k$ goes to $\infty$. The
  topological edges consist of an exponentially growing number of
  legal segments of length $<2$ and the turns between consecutive such
  segments are illegal.
\end{itemize}

Our two main technical lemmas restrict the interval where $A|R_k$ can
make progress in $\S(A)$. Lemma \ref{3 intervals} says:
\begin{itemize}
\item If $A|R_k$ has an illegal turn in a topological edge or a vertex
  with one gate, then the image of $A|R_t$, $t\in (-\infty,k]$, in
    $\S(A)$ is uniformly bounded.
\item If $A|R_k$ has a legal segment of length 2 in a topological
  edge, then the image of $A|R_t$, $t\in [k,\infty)$, in $\S(A)$ is
    uniformly bounded.
\end{itemize}

An easy computation shows that the first bullet holds for $k=-1$ and
the second for $k=9$. Thus all the progress in $\S(A)$ is restricted
to the interval $[-1,9]$. Lemma \ref{uniform crossing} shrinks this
interval further: if every edge of $R_k$ has at least two lifts to
$A|R_k$ then the image of $A|R_t$, $t\in [k,\infty)$, in $\S(A)$ is
  uniformly bounded. In our example this holds for $k=7$.

To simplify the analysis, we will work with {\it greedy} folding paths
as we did in \cite{F}. The above path (an axis of $\Phi$) is {\it not}
a greedy folding path since not all illegal turns are folding. For
example, we will argue in the proof of Lemma \ref{3 intervals} that if
$A|G_{t_0}$ has an interior illegal turn then the same is true for all
$A|G_t$ for $t<t_0$. Note that this fails for our path $R_k$.

We conclude the paper with the Bounded geodesic image theorem,
analogous to the theorem of Masur and Minsky \cite{MM2}.

\vskip 0.3cm
\noindent
{\bf Theorem 5.3} (Bounded geodesic image theorem).
{\it Let $(A_i)$ be a geodesic in $\F$ and $A\in\F$ at distance $>4$
from each $A_i$. Then the set of projections of the $A_i$'s to $\S(A)$
is a uniformly bounded set.}
\vskip 0.3cm

Sabalka and Savchuk in \cite{SS} define a different notion of
projection in the language of sphere systems. It is a partially
defined map from the sphere complex of $M=\#_1^n S^1\times S^2$ to the sphere
complex of a submanifold of $M$ which is a complementary component of
a sphere system. They also prove versions of the Behrstock
inequality and the Bounded geodesic image theorem. It would be
interesting to understand the precise relationship between the two
notions of projection.

A technical tool we use is that folding paths in Outer space project
to reparametrized quasi-geodesics in the splitting complex $\S$. In
\cite{S} Handel and Mosher use a special kind of folding paths
(maximal folds on maps with at least 3 gates at every vertex) that
does not suffice for our purposes. In the appendix we follow the same
outline as in \cite{S} but we use (liberal) folding
paths. Interestingly, the proof becomes simpler since we do not have
to arrange the 3 gates condition after every modification.

{\bf Acknowledgements.} We thank Ken Bromberg, Ursula Hamenst\"adt,
Ilya Kapovich and Patrick Reynolds for their interest and useful
comments.

\section{Review}\label{review}

In this section we recall basic definitions and facts about Outer
space with an emphasis on folding paths. For more information see
\cite{F}. 

A {\it graph} is a 1-dimensional CW complex. The {\it rose} in rank
$n$ is the graph $R_n$ with 1 vertex and $n$ edges. A finite connected
graph is a {\it core graph} if it has no valence 1 vertices. The {\it
  core} of a connected graph $\Gamma$ with finitely generated
$\pi_1(\Gamma)$ is the unique core graph $C\subset\Gamma$ so that
inclusion is a homotopy equivalence.

A {\it marking} of a graph $\Gamma$ is a homotopy equivalence
$f:R_n\to\Gamma$. A {\it metric} on a graph $\Gamma$ is a function
that assigns a positive number $\ell(e)$, called {\it length}, to each
edge $e$ of $\Gamma$. A metric allows one to view $\Gamma$ as a path
metric space. Culler-Vogtmann's {\it Outer space} $\X$ is the space of
equivalence classes of triples $(\Gamma,\ell,f)$ where $\Gamma$ is a
finite graph with no vertices of valence 1 or 2, $f$ is a marking, and
$\ell$ is a metric on $\Gamma$ with volume (i.e.\ sum of the lengths
of all edges) equal to 1. Two such triples $(\Gamma,\ell,f)$ and
$(\Gamma',\ell',f')$ are equivalent if there is an isometry
$\phi:\Gamma\to\Gamma'$ such that $\phi f\simeq f'$ (homotopic).

The group $Out(\FF_n)$ can be identified with the group of
homotopy equivalences on $R_n$. The action of $Out(\FF_n)$ on
$\X$ (on the right) is defined by changing the marking:
$[(\Gamma,\ell,f)]\cdot\Phi=[(\Gamma,\ell,f\Phi)]$.

We frequently simplify the notation and replace the equivalence class
$[(\Gamma,\ell,f)]$ by $\Gamma$.

The set $\mathcal C$ of nontrivial conjugacy classes in $\FF_n$ can be
identified with the set of homotopy classes of immersed loops in
$R_n$. If $\alpha\in\mathcal C$ and $[\Gamma,\ell,f)]\in\X$, we define
$\alpha|\Gamma$ to be the immersed loop homotopic to $f(\alpha)$, and
we let $\ell_\Gamma(\alpha)$ be the length of $\alpha|\Gamma$. We then
have $\ell_{\Gamma\Phi}(\alpha)=\ell_\Gamma(\Phi(\alpha))$. 

If $\Gamma,\Gamma'\in\X$ we say that a map $\phi:\Gamma\to\Gamma'$ is
a {\it difference of markings} if $\phi f\simeq f'$. We will normally
consider only maps that are linear (constant speed) on each
edge. Denote by $\sigma(\phi)$ the Lipschitz constant of $\phi$,
i.e. the maximum speed over all edges of $\Gamma$. The Lipschitz
distance is defined
by $$d(\Gamma,\Gamma')=\min_{\phi}\log\sigma(\phi)$$
where the minimum is taken over all difference of markings maps
$\phi:\Gamma\to\Gamma'$.

A {\it train track structure} on a graph $\Gamma$ consists of an
equivalence relation on the set of directions out of each vertex of
$\Gamma$. One usually pictures graphs with a train track structure so
that equivalent directions are tangent.  Equivalence classes are
called {\it gates}.  The tension graph $\Delta_\phi$ has a natural
train track structure, where two directions $d_1,d_2$ out of a vertex
$v\in\Delta_\phi$ are equivalent if they are mapped by $\phi$ to the
same direction out of $\phi(v)$ (which is not required to be a
vertex). We say $\phi$ is {\it optimal} if $\Delta_\phi=\Gamma$ and
there are at least two gates at every vertex of $\Gamma$.  A turn
(i.e. a pair of distinct directions out of a vertex) is {\it illegal}
if the directions are in the same gate, and otherwise it is {\it
  legal}. An immersed path or a loop is {\it legal} if it takes only legal
turns.

An optimal map $\phi:\Gamma\to\Gamma$ is a {\it train track map} if there
is a train track structure on $\Gamma$ which is preserved by $\phi$,
in the sense that legal paths are mapped to legal paths.

Now let $\phi:\Gamma\to\Gamma'$ be optimal. We define a {\it (greedy)
  folding path} induced by $\phi$ as follows. It is the unique (up to
reparametrization) path $G_t$, $t\in [\alpha,\omega]$ with
$G_\alpha=\Gamma$, $G_\omega=\Gamma'$, and for every $t\in
[\alpha,\omega)$ there is $\epsilon>0$ such that for every $\tau\in
  [t,t+\epsilon]$ the graph $G_\tau$ is obtained from $G_t$ by
  identifying segments of length $\tau-t$ leaving any vertex in
  equivalent directions and then rescaling to restore volume 1. 
We normally parametrize the path by arclength, i.e.\ so
  that for $t_1<t_2$ we have $d(G_{t_1},G_{t_2})=t_2-t_1$. 
There are
  induced optimal maps $G_{t_1}\to G_{t_2}$ for $t_1<t_2$ and they
  compose naturally. 

If $\phi:G\to G'$ is any difference of markings map there is an {\it
  induced} path from $G$ to $G'$ as follows. It is a concatenation of
a linear path within a simplex from $G\to G''$ and a folding
path from $G''$ to $G'$ induced by an optimal map (see \cite{francaviglia-martino}, also
\cite[Proposition~2.7]{F}).

Let $G_t$ be a folding path defined on a closed interval
$[\alpha,\omega]$. Then the interval can be subdivided into finitely
many subintervals $[t_i,t_{i+1}]$ so that for every $t\in
(t_i,t_{i+1})$ all illegal turns are at vertices with two gates, with
one gate consisting of a single direction. Moreover, all $G_t$ in this
interval are canonically homeomorphic with the homeomorphism
preserving the train track structure. As $t$ increases in this open
interval, $G_t$ moves along a straight line in the interior of a
simplex in $\X$, i.e.\ only the lengths of the edges change. See
Figure~\ref{no event}.

\begin{figure*}[h]
\centering
\includegraphics[scale=1]{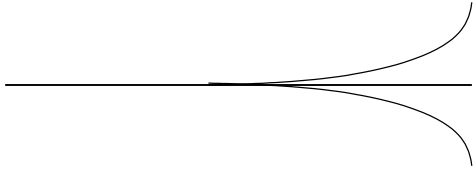}
\caption{Folding path between events}
\label{no event}
\end{figure*}

The proof of these facts follows from the analysis of folding paths in
\cite{F}. We understand what happens to the right of a given $G_t$, by
the definition of a folding path.  To understand what happens on an
interval $[t_0-\epsilon,t_0]$ we use the local description of
unfolding in \cite[Section 4]{F}.

Let $N$ be a small connected neighborhood of a vertex
$v$ in $G_{t_0}$. Thus $N$ is a wedge of arcs along their
endpoints, with $v$
the wedge point. For small $\epsilon>0$ the preimage $N_\epsilon\subset
G_{t_0-\epsilon}$ of $N$ is a tree that comes with a height function
and a collection of {\it widgets}. See Figure \ref{widgets}.

\begin{figure*}[h]
\centering
\includegraphics[scale=0.5]{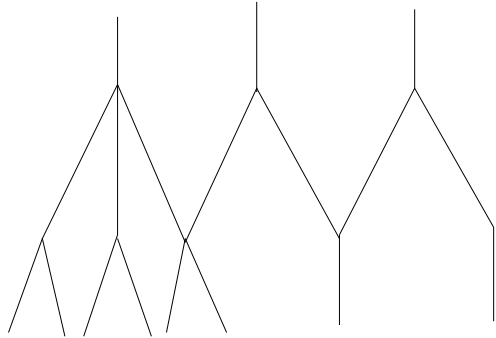}
\caption{The structure of $N_\epsilon$. This one has 3 widgets and 5
  vertices at height 0.}
\label{widgets}
\end{figure*}

The preimages of $v$ are at height 0. All illegal turns are at the
vertices at height $\epsilon$. Each of these vertices has two gates,
one gate consists of a single direction drawn upwards and the other
gate has $\geq 2$ directions all of which fold together at $t=t_0$. The
union of the edges in the downwards gate at each vertex at height
$\epsilon$ is a widget. The union of all widgets is a tree.

We also recall that there are coarse projection maps $\X\to\S\to\F$
defined by sending a graph $\Gamma\in\X$ to a 1-edge splitting of
$\FF_n$ obtained by collapsing all but one edge of $\Gamma$, and
sending a 1-edge splitting to a vertex stabilizer. A useful fact,
supplementing Theorems~\ref{bf} and \ref{hm2}, is that the image of a
folding path in $\X$ is a reparametrized quasi-geodesic in both $\F$
and $\S$, with uniform constants (see \cite{F} for $\F$ and the
appendix of this paper for $\S$).

\section{Folding path lemmas}

Recall that $A|G$ denotes the core of the covering space of $G$
corresponding to a subgroup $A$ of $\FF_n$. 
In this section we prove two technical lemmas about induced
paths. The setting in both is that
$G_t$ is a folding path in $\X$, $A|G_t$ is the induced path in Outer
space $\X(A)$ of $A$, and we are interested in identifying
geometrically the portions of the path $A|G_t$ where it (possibly)
makes progress in the splitting complex $\S(A)$ of $A$. 

\vskip 6pt \noindent{\it Convention.} For the rest of this paper we use {\it bounded} to mean {\it bounded by a function of $rank(\FF_n)$ alone.}\vskip 6pt

By a {\it topological ({\rm or} natural) edge} in a core graph we mean
an edge with respect to the cell structure with all vertices of
valence $\geq 3$, and likewise, vertices with valence $\geq 3$ are
{\it natural} (when the core graph is a circle we don't talk about
topological edges). An {\it interior illegal turn} is the turn (pair of
directions) at a valence two vertex with one gate. Such turns arise in
$A|G_t$.

\begin{lemma}\label{3 intervals}
Given a folding path $G_t$, $t\in [\alpha,\omega]$, and a f.g.\ subgroup $A<\FF_n$, the path
$A|G_t$ can be subdivided into 3 segments
$[\alpha,\beta),[\beta,\gamma),[\gamma,\omega]$ (with one or two
segments possibly omitted) so that:
\begin{enumerate}[(1)]
\item On the first segment $[\alpha,\beta)$ $A|G_t$ has a natural
  vertex with one gate or an interior illegal turn.
\item On the middle segment $[\beta,\gamma)$ $A|G_t$ has all edges of
  size $\leq 2$, hence it has bounded volume, all vertices have $\geq
  2$ gates and there are no interior illegal turns.
\item On the last segment $[\gamma,\omega]$ $A|G_t$ has a legal
  segment of length $\geq 2$ contained in a topological edge.
\end{enumerate}
Moreover, the image of the first and the last segment in the splitting
complex $\S(A)$ of $A$ is bounded.
\end{lemma}

\begin{proof}
Denote by $J$ the set of those $t_0\in [\alpha,\omega]$ such that
$A|G_{t_0}$ contains a vertex with one gate or an interior illegal
turn. We will first argue that $J$ is open in $[\alpha,\omega]$. It is
clear from the definition of folding paths that if $t_0\in J$ and
$t_0\neq\omega$ then $[t_0,t_0+\epsilon)\subset J$ for small
  $\epsilon$. Indeed, the effect of local folding at a valence $k$
  vertex of $A|G_{t_0}$ with one gate is that an edge with a vertex of
  valence 1 develops, and after removing it (which is required when
  passing to the core) a valence $k$ vertex with one gate remains.
  Assume now that $t_0\in J$ and $t_0\neq\alpha$. We will show that
  for sufficiently small $\epsilon>0$ we have
  $[t_0-\epsilon,t_0]\subset J$ and this will complete the proof that
  $J$ is open. The argument uses the description of unfolding in
  Section \ref{review} and will at the same time lay the groundwork
  for showing that on $J$ the path is not making progress in $\S(A)$. 

Let $v$ be a vertex of $A|G_{t_0}$ of valence $k$ and with one gate
  (so $k=2$ in the case of an interior illegal turn).
There are
  $k$ free splittings of $A|G_{t_0}$ obtained by collapsing all edges
  of $A|G_{t_0}$ except for one edge incident to $v$ (some of these
  splittings may coincide). These splittings are always uniformly
  close in $\S(A)$. We will further argue that either
\begin{enumerate}
\item [(A1)]
$A|G_{t_0-\epsilon}$ also has a vertex of
valence $k$ and one gate and the associated splittings are the same as
at time $t_0$, or
\item [(A2)] $A|G_{t_0-\epsilon}$ has a vertex of valence $<k$ and one gate and
  the associated splittings are uniformly close to the ones at time
  $t_0$ (this is of course possible only if $k>2$).
\end{enumerate}

Let $N$, $N_\epsilon$ be as in Section \ref{review}, i.e.\ $N$ is a
small neighborhood of the image of $v$ in $G_{t_0}$ and $N_\epsilon$
is the preimage of $N$ in $G_{t_0-\epsilon}$. Denote by $\tilde G_t$
the covering of $G_t$ corresponding to $A$; thus $A|G_t\subset \tilde
G_t$ is the core. Let $\tilde N$ be the lift of $N$ to $\tilde
G_{t_0}$ that contains $v$ and $\tilde N_\epsilon$ the preimage of
$\tilde N$ under the lifted folding map $\tilde
G_{t_0-\epsilon}\to\tilde G_{t_0}$. Since $\tilde N_\epsilon\to
N_\epsilon$ is a homeomorphism, we will talk about height, widgets,
etc in $\tilde N_\epsilon$. All edges in a widget in $\tilde
N_\epsilon$ map to the same small segment in $\tilde G_{t_0}$ with one
endpoint $v$. If this segment is in the core $A|G_{t_0}$ we will call
the widget {\it essential}, and otherwise it is {\it
  inessential}. Likewise, an edge below height 0 is essential if its
image is in $A|G_{t_0}$, otherwise it is inessential. Altogether,
there are precisely $k$ essential edges and widgets.  We now observe:
\begin{enumerate}
\item [(B1)]
If $x$ is a vertex of $\tilde N_\epsilon$ at height $\epsilon$ that belongs
to an inessential widget, then the edge that goes upwards from $x$
does not belong to $A|G_{t_0-\epsilon}$.
\item [(B2)] Essential widgets and essential edges are pairwise disjoint.
\item [(B3)]
If $z$ is a vertex of $\tilde N_\epsilon$ at height 0 that belongs to an
essential widget, then any edges that go down from $z$ do not belong
to $A|G_{t_0-\epsilon}$. 
\end{enumerate}
All of these statements are consequences of the facts that $v$ has only one gate
in $A|G_{t_0}$ and that the turns at $z$ and at $x$ involving the
upward edge are legal.

Essential edges below height 0 are the only ones that map to their
respective segments adjacent to $v$ so they are necessarily part of
the core $A|G_{t_0-\epsilon}$. For the same reason, at least one edge
in each essential widget belongs to $A|G_{t_0-\epsilon}$.  Denote by
$\tilde N'_\epsilon$ the hull of all essential edges and widgets, together
with edges above essential widgets. The rest of $\tilde N_\epsilon$ is not
part of the core (and even some of the edges in essential widgets
might be outside the core). We retain the terminology from
$\tilde N_\epsilon$, e.g. an inessential widget in $\tilde N'_\epsilon$ is the
intersection of an inessential widget in $\tilde N_\epsilon$ with
$\tilde N'_\epsilon$. But note that in $\tilde N'_\epsilon$:
\begin{enumerate}[(C1)]
\item
There are no edges above inessential
widgets. 
\item There are no edges below essential widgets.
\item
Each inessential widget contains at least 2 edges. 
\item
$\tilde N'_\epsilon$ must contain at least one inessential widget.
\item Each complementary component of the height $\epsilon$ vertex in
  each inessential widget contains at least one essential edge or
  essential widget. 
\item Every inessential widget has $\leq k$ edges.
\end{enumerate}

(C4) follows from the fact that essential widgets are pairwise
disjoint and it's impossible for $\tilde N'_\epsilon$ to consist just of a
single essential widget. (C5) follows from the construction of
$\tilde N'_\epsilon$, and (C6) follows from (C5).

If $\tilde N'_\epsilon$ has an inessential widget with $<k$ edges, we are
done (the associated splittings are clearly within bounded distance of
the original $k$ splittings). Suppose $W$ is an inessential widget in
$\tilde N'_\epsilon$ with exactly $k$ edges. We will argue that the $k$
splittings associated to this widget are exactly the same as the
original $k$ splittings. All other inessential
widgets in $\tilde N'_\epsilon$ have two edges (see (C5)). Thus each edge
of $W$ is attached to a (possibly empty) sequence of inessential
widgets with 2 edges and then either an essential edge or an essential
widget. The map $\tilde N'_\epsilon\to \tilde N$ factors through
$\tilde N''_\epsilon$
obtained from $\tilde N'_\epsilon$ by folding together all edges in each
essential widget. The folding map does not affect the associated
splittings of $A$. Now observe that $\tilde N''_\epsilon$ is a $k$-pronged star
(wedge of $k$ arcs along their endpoints) and the map $\tilde
N''_\epsilon\to \tilde N'_\epsilon$
amounts to partially folding together the $k$ prongs and this clearly
does not affect the splittings.

In Figures \ref{case1}, \ref{case2} and \ref{case3} we illustrate the
3 possibilities when $k=2$.

\begin{figure*}[h]
\centering
\includegraphics[scale=0.5]{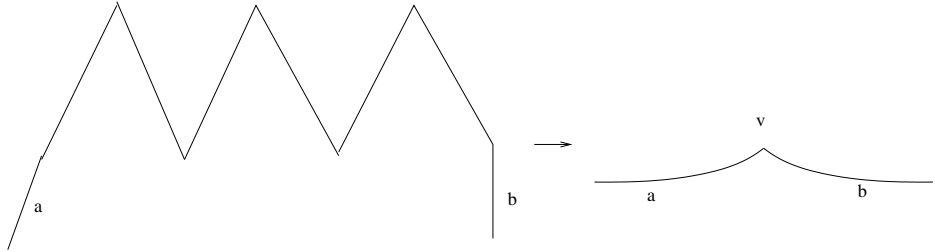}
\caption{This is $\tilde N'_\epsilon$ for the example in Figure \ref{widgets}
assuming there are two essential edges, the leftmost and the rightmost.}
\label{case1}
\end{figure*}

\begin{figure*}[h]
\centering
\includegraphics[scale=0.5]{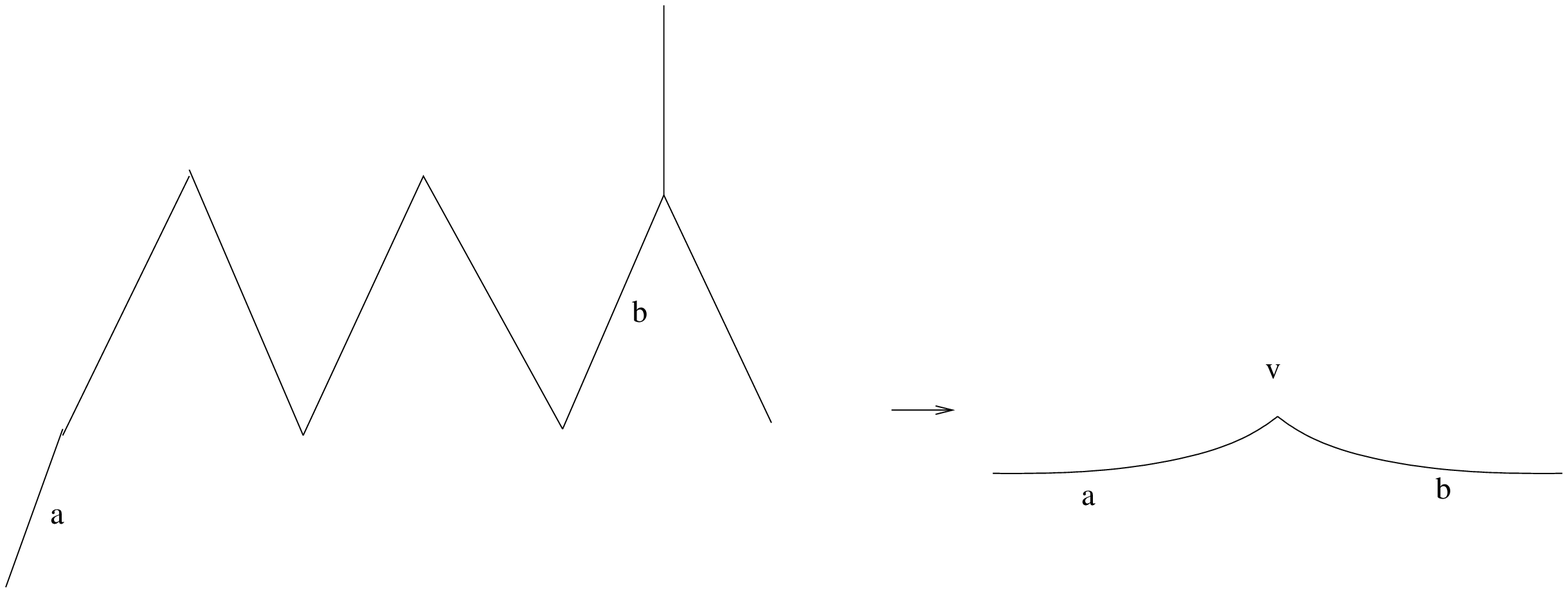}
\caption{This is $N'_\epsilon$ for the example in Figure \ref{widgets}
assuming that the leftmost edge and the rightmost widget are essential.}
\label{case2}
\end{figure*}

\begin{figure*}[h]
\centering
\includegraphics[scale=0.5]{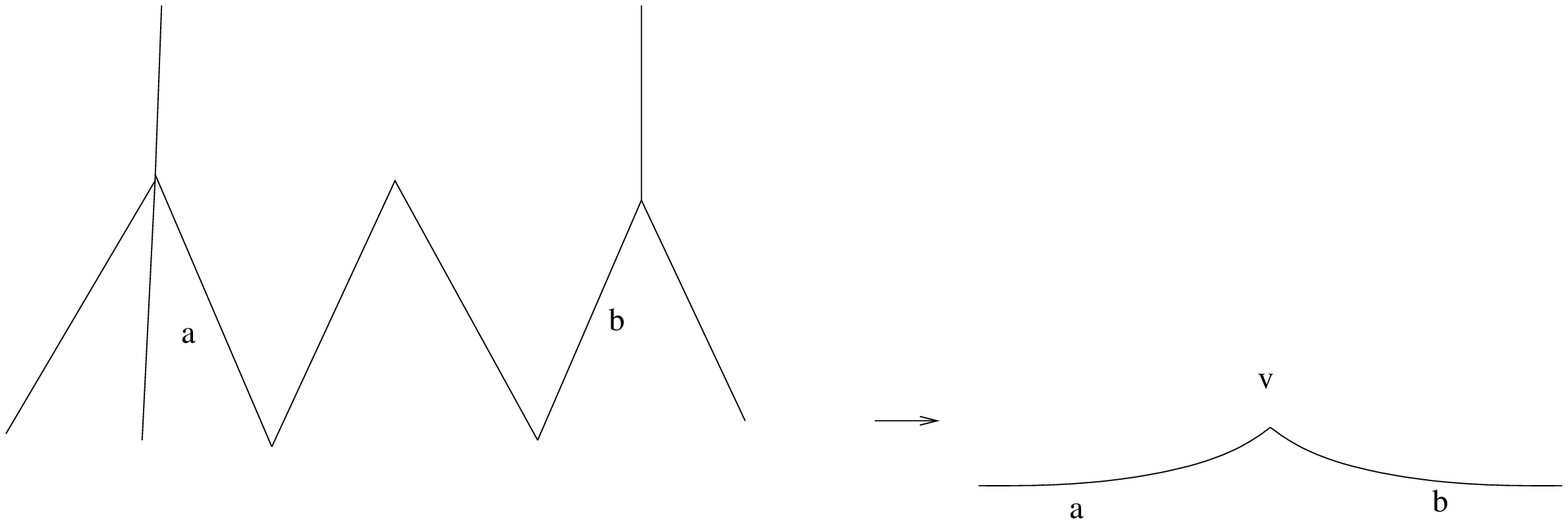}
\caption{This is $N'_\epsilon$ for the example in Figure \ref{widgets}
assuming that the leftmost and the rightmost widgets are essential.}
\label{case3}
\end{figure*}

We now know that $J$ is open. Next, we observe that $J$ must be an
initial segment of $[\alpha,\omega]$ (or empty). Indeed, otherwise we
would have $t_0\not\in J$, but $(t_0,t_0+\epsilon)\subset J$. However,
a local fold at $A|G_{t_0}$ cannot produce a vertex with one
gate. Similarly, splittings visible at $t_0\in J$ are also visible on
$[t_0,t_0+\epsilon]$.
The picture is that as $A|G_t$ unfolds
(with $t$ decreasing),
our valence $k$ vertex with one gate undergoes (perhaps
simultaneously) a sequence of changes
where it lowers the valence, but retains the feature that it has only
one gate. Between these events, the associated splittings do not change,
and at the events they change but only by a bounded distance.

Now (A1), (A2) and this discussion then
prove that we can take $[\alpha,\beta)=J$ and the image of $G_t$,
  $t\in J$ is bounded in
$\S(A)$.

To conclude the proof of Lemma \ref{3 intervals} we have to consider
the last segment. Suppose $A|G_{t_0}$ has a legal segment of length
$\geq 2$ in a topological edge. Consider the splitting of $A$ given by
collapsing all other topological edges of $A|G_{t_0}$. We must show
that for $t>t_0$ the graph $A|G_{t_0}$ has a topological edge that
contains a legal segment of length $\geq 2$ and that the corresponding
splitting coincides with the original.

For this argument it is convenient to work in the unprojectivized
Outer space and use the natural parametrization of the folding path,
which we rename $H_t$, so
that all folds have speed 1. We also take $G_{t_0}$ as the starting
point $H_0$. Thus after time $\tau$ the volume of $H_{\tau}$ is $\leq
1-\tau$ (equality holds if there is only one illegal turn the whole
way, otherwise the volume decreases faster). In particular,
$\tau<1$. The legal segment in $A|H_0$ of length $L\geq 2$ may lose a
piece of length $\tau$ at each end, so we have a legal segment of
length $\geq L-2\tau$ in a
topological edge in $A|H_\tau$. After rescaling by $\geq \frac
1{1-\tau}$ to bring $H_\tau$ to volume 1, the legal segment has length
$\geq \frac {L-2\tau}{1-\tau}\geq L$ when $L\geq 2$. Moreover, the
midpoint of the legal segment is never identified with any other
point, so the splitting it determines is independent of $\tau$.
\end{proof}

\begin{cor}\label{quasi-geodesic}
Let $G_t$ be a folding path in $\X$. For any f.g.\ subgroup $A$ of
$\FF_n$ the image of $A|G_t$ is a reparametrized quasi-geodesic in
$\S(A)$ with
uniform constants.
\end{cor}

\begin{proof}
Progress in $\S(A)$ is made only on the middle interval, and on this
interval $A|G_t$ is a folding path in Outer space of $A$ (after
rescaling to achieve volume 1, and reparametrizing the path so it is
unit speed). The statement then follows from the fact that
folding paths project to reparametrized quasi-geodesics.
\end{proof}

Let $G_t$ be a folding path. We say that a collection of subgraphs
$H_t\subset G_t$ is {\it forward invariant} if for every $t<t'$ the
folding map $G_t\to G_{t'}$ takes $H_t$ into $H_{t'}$.

We will denote by $\Theta_t$ the image of $A|G_t\to G_t$, by
$\Omega_t\subset G_t$ the set of points with at least two preimages in
$A|G_t$, and by $\tilde\Omega_t\subset A|G_t$ the preimage of
$\Omega_t$ in $A|G_t$.

From now on we restrict $G_t$ to the middle segment of Lemma \ref{3
  intervals}. 

\begin{prop} \label{details}
Assume $G_t$ consists only of the middle segment.
\begin{itemize}
\item $\Theta_t$ is a forward invariant core graph
  whose rank never decreases, and it may increase a bounded number of
  times,
\item both $\Omega_t$ and $\tilde\Omega_t$ are forward invariant and
  have a bounded number of components,
\item both $\Omega_t$ and $\tilde\Omega_t$ can change their homotopy type
  only a bounded number of times, each time by one of the following
  events (more than one event may occur at the same time):
\begin{itemize}
\item the rank of a component increases, or two components merge
into one,
\item an isolated point is created. In this case the rank of
  $\Theta_t$ 
increases.
\end{itemize}
\end{itemize}
\end{prop}

\begin{proof}
Let $\beta=t_0<t_1<\cdots<t_M=\gamma$ be the subdivision of
$[\beta,\gamma]$ as in Section \ref{review} with the $t_i$'s
corresponding to the ``events'' in the folding path. First consider
the case
$t\in (t_i,t_{i+1})$.
The image $\Theta_t$ is a core graph and locally near a vertex
as in Figure \ref{no event} it could either be empty or it contains
the edge pictured to the left of the vertex as well as at least one
edge to the right. The quality of the image does not change as $t$
increases in $(t_i,t_{i+1})$. 

The subgraph $\Omega_t$ near the vertex above could be empty, or it
could contain the left edge and 0 or more right edges. Its quality
also does not change in the open interval $[t_i,t_{i+1})$. See Figure
\ref{omega}. 

\begin{figure*}[h]
\centering
\includegraphics[scale=1]{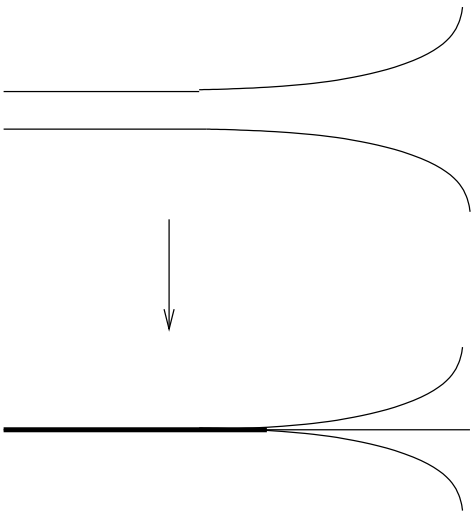}
\caption{$\Omega_t$ is locally just the left edge}
\label{omega}
\end{figure*}

Thus the homotopy type of $\Theta_t$ and of $\Omega_t$ is constant for
$t\in [t_i,t_{i+1})$ and can change only at the discrete times
  $t=t_i$, and the same is true for $\tilde\Omega_t$. Now let us
  analyze a small interval $[t_i-\epsilon,t_i]$ as in Section
  \ref{review}.  Take a small connected neighborhood $N$ of a vertex
  in $G_{t_i}$ and look at its preimage $N_\epsilon$ in
  $G_{t_i-\epsilon}$. It is a tree with the additional structure of a
  height function and widgets. Observe the following:
\begin{itemize}
\item If $\Theta_{t_i-\epsilon}\cap
  N_\epsilon=\emptyset$ then $\Theta_{t_i}\cap
  N=\emptyset$.
\item If $\Theta_{t_i-\epsilon}\cap
  N_\epsilon\neq\emptyset$, at $t=t_i$ all the components of this
  intersection merge
  together. So if the intersection is connected the image does not
  locally change the homotopy type at $t=t_i$ and otherwise the Euler
  characteristic decreases, and the change amounts to attaching 1-cells.
\item If $\Omega_{t_i-\epsilon}\cap N_\epsilon\neq\emptyset$, at
  $t=t_i$ all the components of the intersection merge together. So if
  the intersection is connected $\Omega_t$ does not locally change the
  homotopy type and otherwise the Euler characteristic decreases, and
  the change amounts to attaching 1-cells. 
\item If $\Omega_{t_i-\epsilon}\cap N_\epsilon=\emptyset$ but
  $\Theta_{t_i-\epsilon}\cap N_\epsilon\neq\emptyset$ there are two
  possibilities. If $\Theta_{t_i-\epsilon}\cap N_\epsilon$ is
  connected, then this intersection lifts homeomorphically to $A|G_t$
  (if two edges of $\Theta_{t_i-\epsilon}$ share a vertex which is not
  in $\Omega_{t_i-\epsilon}$, their unique lifts must also share a
  vertex). Thus in this case $\Omega_{t_i}\cap N=\emptyset$ as
  well. The other possibility is that $\Theta_{t_i-\epsilon}\cap
  N_\epsilon$ is disconnected. In that case the Euler characteristic
  of the image decreases at $t=t_i$ as noted above, but note that
  $\Omega_{t_i}\cap N$ is a single point, i.e. the Euler
  characteristic of $\Omega_t$ increases by 1 at $t=t_i$ (the lifted
  components cannot merge at $t=t_i$ as that would change the homotopy
  type of $A|G_{t}$ at $t=t_i$).
\end{itemize}

The set of $t=t_i$'s where $\chi(\Theta_t)$ decreases is uniformly
bounded, and this set contains all events when $\Omega_t$ gains an
isolated point. The set of additional times when $\chi(\Omega_t)$
decreases is also bounded, since $\Omega_t$ is a subgraph (with
respect to the natural $CW$ structure) of $G_t$ and has a uniformly
bounded number of components. Thus on the intervals between these
times the homotopy types of $\Theta_t$ and of $\Omega_t$ are
constant. We now claim that between these times the homotopy type of
$\tilde\Omega_t$ is also constant.

To prove the claim, we again analyze an interval $[t_i-\epsilon,t_i]$
as above. We may assume that $\Theta_{t_i-\epsilon}\cap N_\epsilon$
and $\Omega_{t_i-\epsilon}\cap N_\epsilon$ are both connected
(otherwise a change in the homotopy type of $\Theta_t$ or $\Omega_t$
occurs at $t=t_i$) and nonempty. In the analysis we will use the fact
that we are in the middle interval, so in particular every vertex of
$A|G_t$ has at least two gates. Every component of the preimage of $N$
(a small metric neighborhood of a vertex $v$ in $G_{t_i}$) is a metric
neighborhood of a point of $\tilde G_{t_i}$, the $A$-cover of
$G_{t_i}$. These components are metric neighborhoods of vertices, and
when intersected with $A|G_{t_i}$ become metric neighborhoods of
points in $A|G_{t_i}$ (or empty). Some of these intersections are
contained in the interior of a topological edge (we will view them as
degenerate); the others are small metric neighborhoods $\tilde
N(\tilde v_k)$ of vertices $\tilde v_k$ of $A|G_{t_i}$ that map to
$v$. We now have a commutative diagram
\[
\begin{matrix}
\tilde N_\epsilon(\tilde v_k)&\longrightarrow&\tilde N(\tilde v_k)\\
\downarrow&&\downarrow\\
N_\epsilon&\longrightarrow&N
\end{matrix}
\]
where $\tilde N_\epsilon(\tilde v_k)$ is the full preimage of $\tilde
N(\tilde v_k)$ in $A|G_{t_i-\epsilon}$, and it is also a component of
the preimage of $N_\epsilon$. Note also that $\tilde N_\epsilon(\tilde
v_k)$ carries the structure of widgets, since $\tilde v_k$ has at
least two gates, and that the left vertical arrow is an embedding. The
union of the images of the components $\tilde N_\epsilon(\tilde
v_k)$'s in $N_\epsilon$ (including the degenerate components) is
$\Theta_{t_i-\epsilon}\cap N_\epsilon$, and pairwise intersections of
the images of components are contained in $\Omega_{t_i-\epsilon}\cap
N_\epsilon$. It follows that the image of each component intersects
$\Omega_{t_i-\epsilon}\cap N_\epsilon$ in a connected set,
i.e. $\tilde\Omega_{t_i-\epsilon}$ intersected with each component of
the preimage of $N_\epsilon$ is nonempty and connected (so the
degenerate components are contained in
$\tilde\Omega_{t_i-\epsilon}$). Thus $\tilde\Omega_t$ does not change
its homotopy type in the limit.

Finally, we argue that $\tilde\Omega_t$ has a uniformly bounded number
of components. If not, some topological edge $E$ of $A|G_t$ must intersect
$\tilde\Omega_t$ in a large number of components. Each point in the
frontier of the intersection maps to one of boundedly many vertices in
$G_t$, and each vertex has bounded valence, so we find distinct
components in $E-\tilde\Omega_t$ with overlapping images,
contradicting the definition of $\tilde\Omega_t$.
\end{proof}

\section{Subfactor projections}

\subsection{Bounding distances in $\S$}

If $G\in\X$, we may collapse all edges of $G$ except for one and
obtain a splitting. Choosing different edges gives splittings at
distance 1 from each other, so we have an equivariant coarse
projection map
$\X\to\S$.

\begin{lemma}\label{distance}
Let $\phi:G\to G'$ be a difference of markings so that the preimage of
a point has cardinality $k$. Then $d_\S(G,G')$, i.e.\ the distance in
$\S$ between the projections of $G,G'$, is bounded by a linear
function of $k$. Moreover, if $G_t$ is the path induced by $\phi$ (see
Section~\ref{review}) then the image of the whole path in $\S$ is also
bounded by a linear function in $k$.
\end{lemma}

\begin{proof}
$G_t$ is the concatenation of a linear path in a simplex and a folding
  path. We may assume that the linear path is trivial. By perturbing,
  we may also assume the point $x'\in G'$ is not a vertex (perturbing
  slightly can increase cardinality by at most a bounded factor), and
  that $\phi^{-1}(x')$ does not contain a vertex. If $k=1$ the two
  points $x'\in G'$ and $x=\phi^{-1}(x')$ define the same splitting
  (collapse the complements of small neighborhoods of $x,x'$ on which
  $\phi$ is a homeomorphism). The same holds $x_t$, the preimage of
  $x'$ in $G_t$. Otherwise, subdivide $G_t$ so that on the first
  portion the preimage consists of $k$ points, and on the second
  portion it consists of $<k$ points. Then the second portion has
  bounded image by induction, and on the first portion the case $k=1$
  applies.
\end{proof}

\begin{lemma}\label{distance 2}
Suppose $\phi:G\to G'$ is optimal and there are proper core graphs
$H\subset G$ and $H'\subset G'$ so that $\phi(H)\subset H'$ and
$\phi:H\to H'$ is a homotopy equivalence. If there is a point $x'\in
H'$ such that $\phi^{-1}(x')\subset H$ then $d_\S(G,G')$ is bounded.
\end{lemma}

\begin{proof}
Fold $H$ until $H=H'$ and $\phi$ is identity on $H$. This process does
not affect edges in the complement of $H$ so there is no movement in
$\S$. At the end $\phi^{-1}(x')$ consists of a single point. Now apply
Lemma \ref{distance}.
\end{proof}

\begin{remark}
A similar argument shows that if $\phi^{-1}(x')-H$ has $k$ points then
the distance in $\S$ between $G$ and $G'$ is bounded by a linear
function of $k$.
\end{remark}

\subsection{Restricting progress in $\S$ further}
Recall that a subgroup $A$ of $\FF_n$ is {\it malnormal} if
$gAg^{-1}\cap A\neq \{1\}$ implies $g\in A$.  For example, free
factors are malnormal. The next lemma shows that if $A$ is malnormal,
$A|G_t$ can be making progress only if $\Omega_t\neq G_t$ (see
Proposition~\ref{details}).

\begin{lemma}\label{uniform crossing}
Let $A$ be a malnormal f.g.\ subgroup of $\FF_n$.  In the setting of
Lemma \ref{3 intervals} assume that for some $t_0\in [\beta,\gamma)$
  (the middle interval) every edge $e$ of $G_{t_0}$ is covered by
  $A|G_{t_0}$ at least twice, i.e.\ $\Omega_{t_0}=G_{t_0}$. Then the
  image of $A|G_t$ for $t\in [t_0,\gamma)$ in $\S(A)$ is a uniformly
    bounded set, and so is the image of $G_t$ in $\S$ on the same
    interval.
\end{lemma}

\begin{proof}
Fix $t\geq t_0$ and let $x$ be the midpoint of an edge of $G_t$ of
length $\geq \frac 1{3n-3}$. 
Consider the commutative diagram
\[
\begin{array}{ccc}
A|G_{t_0}&\rightarrow&A|G_t\\
\downarrow&&\downarrow\\
G_{t_0}&\rightarrow&G_t
\end{array}
\]
\noindent
which exists because we are assuming $t_0$ is not in the first
interval $[\alpha,\beta)$, so that $A|G_{t_0}$ maps to the core
  $A|G_t$. Taking the preimage of $x$ in each of the 4 spaces yields
\[
\begin{array}{ccc}
R&\rightarrow&P\\
\downarrow&&\downarrow\\
Q&\rightarrow&x
\end{array}
\]
\noindent
Because we are assuming that $t$ is in the middle interval
$[\beta,\gamma)$ the cardinality of $P$ is uniformly bounded.
We will argue: 

{\it Claim 1.}
The cardinality of the preimage $R_q$ in $R$ of any $q\in Q$ is uniformly
bounded.

{\it Claim 2.}
The cardinality of the preimage $Q$ of $x$ in
$G_{t_0}$ is uniformly bounded.

The claims together imply that $R$ is uniformly bounded.  Lemma
\ref{distance} plus Claim 2 imply that the path from $G_{t_0}$ to
$G_t$ makes bounded progress in $\S$ and the boundedness of $R$
implies (by
applying Lemma \ref{distance} to any point in $P$, which is nonempty
by assumption) that the path from $A|G_{t_0}$ to $A|G_t$ makes bounded
progress in $\S(A)$.

{\it Proof of Claim 1.} From the diagram we have a map $R_q\to
P$. Since $P$ is bounded it suffices to argue that this map is
injective. Suppose $a,b\in R_q$ are distinct but map to the same
point $p\in P$. Choose an immersed path $\tau$ from $a$ to $b$ in
$A|G_{t_0}$. Since $(A|G_{t_0},a)\to (A|G_t,p)$ is an isomorphism in
$\pi_1$, there is a loop $\tau'$ in $A|G_{t_0}$ based at $a$ that maps
to a loop in $A|G_t$ homotopic to the image of $\tau$, rel $p$ and the
same holds in $G_t$ with basepoint $x$. Since $(G_{t_0},q)\to (G_t,x)$
is an isomorphism in $\pi_1$, the images of $\tau$ and $\tau'$ have to
be homotopic rel $q$. Since they are both immersed, they must be
equal, implying that $\tau=\tau'$, contradiction. Alternatively, argue
that the folding process does not identify distinct points of $R_q$.  

{\it Proof of Claim 2.}
Fix an edge $e$ of $G_{t_0}$. We will argue that $e$ intersects
$Q$ in a bounded set. Since the number of edges is bounded, this will
prove the claim.

Let $u,v$ be two distinct lifts of $e$ to $A|G_{t_0}$ (they exist by
assumption).  Define a function $\sigma:e\cap Q\to P^2$ as follows. If
$y\in e\cap Q$, let $y_u,y_v$ be the lifts of $y$ to $u,v$
respectively and let $p(y_u),p(y_v)$ be their images in $P$. Define
$\sigma(y)=(p(y_u),p(y_v))\in P^2$. Since $P$ is uniformly bounded, it
suffices to argue that $\sigma$ is injective.

Assume $\sigma(y)=\sigma(y')$ for $y\neq y'$. Let $I$ be the segment
of $e$ whose endpoints are $y$ and $y'$ and let $I_u,I_v$ be the lifts
of $I$ to $u,v$ respectively. Then the images of $I_u$ and $I_v$ are
loops in $A|G_t$ based at distinct points of $P$ (see the proof of
Claim 1) and they are lifts of the image of $I$. But this violates the
malnormality of $A$.
\end{proof}

\subsection{Subgroups adjoined in a graph}

We say that a f.g.\ subgroup $A<\FF_n$ is {\it adjoined} in a graph
$G\in\X$ if the map $A|G\to G$ is 1-1 over some edge of $G$. If we use
the notation from Proposition \ref{details} and denote by
$\Omega\subset G$ the subgraph over which $A|G\to G$ is at least 2-1
and by $\tilde\Omega$ its preimage in $A|G$, then the condition
amounts to saying $\tilde\Omega\neq A|G$.  In other words, $G$ is then
obtained from some base graph $K$ (equal to $\Omega$ if $A|G\to G$ is
onto) by gluing $A|G$ along a proper subgraph $\tilde\Omega\subset
A|G$ via an immersion $f:\tilde\Omega\to K$, and the cardinality of
$f^{-1}(f(x))$ is $\ge 2$ for all $x\in K$.

We say that $A$ is {\it nearly embedded} in $G$ if $\tilde\Omega$ is a
forest, and it is {\it pinched} if $\tilde\Omega$ is a finite set. We
say $A$ is {\it embedded} in $G$ if $A|G\to G$ is 1-1
(i.e. $\tilde\Omega=\emptyset$). 

\begin{lemma}\label{near embedding}
If $A$ is nearly embedded in $G$ then $A$ is a free factor.
\end{lemma}

\begin{proof}
Let $p:A|G\to G$ be the canonical immersion. Enlarge the forest
$\tilde\Omega$ to a maximal tree $T$ for $A|G$ and let $E$ be the set
of edges of $A|G$ not in $T$. Since edges of $E$ aren't in
$\tilde\Omega$, $E$ is in a 1-1 correspondence with the set of edges
$p(E)$ in $G$. Let $x\in T$ and consider performing Stallings folds to
the morphism $$(T\cup E)\vee_{x=p(x)} (G-p(E))\to G$$ Folding the tree
$T$ into $G-p(E)$ doesn't affect the homotopy type and when finished
the resulting map is 1-1 and onto. In particular, the morphism is a
homotopy equivalence. Since $T\cup E$ represents $A$, the lemma
follows.
\end{proof}

\begin{prop}\label{analysis}
For every folding path $G_t$ there is an interval $[\beta,\tau)$
contained in the middle interval $[\beta,\gamma)$ (see Lemma~\ref{3
    intervals}) so that:
\begin{itemize}
\item $A$ is adjoined in $G_t$ for all $t\in [\beta,\tau)$, and
\item both complementary components of $[\beta,\tau)$ in
  $[\alpha,\omega]$ have uniformly
  bounded projections in $\S(A)$.
\end{itemize}
\end{prop}

Thus informally, a folding path can make progress in $\S(A)$ only
while $A$ is adjoined in the middle interval.

\begin{proof}
First, by Lemma~\ref{3 intervals} the path makes no progress outside
the middle interval $[\beta,\gamma)$. Let $[\beta,\tau)$ be the
    interval inside $[\beta,\gamma)$ consisting of $t\in
      [\beta,\gamma)$ such that $A$ is adjoined in $G_t$. This is an
        initial interval (possibly empty) by the forward invariance of
        $\tilde\Omega_t$, see Proposition~\ref{details}. It remains to
        argue that on $[\tau,\gamma)$ the path makes bounded progress
          in $\S(A)$. On this interval we have $\tilde\Omega_t=A|G_t$,
          so if $A|G_t\to G_t$ is onto the claim follows from
          Lemma~\ref{uniform crossing}.

Otherwise, the image is a core graph $\Theta_t$ and we may assume it
has constant rank by subdividing into a bounded number of subpaths and
working with each subpath separately. The volume of $\Theta_t$ is
$\leq 1$;
after rescaling so that $vol(\Theta_t)=1$ and reparametrizing the path
so that folds occur at speed 1, we have a folding
path with $A|\Theta_t\to \Theta_t$ surjective, so we apply the same
argument. 
\end{proof}

\begin{notation}\label{notation}
If $A$ is adjoined in $G$, we sometimes use the following
notation. The immersion $A|G\to G$ is denoted $p$. Enlarge a maximal
forest for $\tilde\Omega$ to a maximal tree $T$ for $A|G$. The set of
edges of $A|G-T$ is the disjoint union of a set $E_{\tilde\Omega}$ of
edges in $\tilde\Omega$ and a set $E$ of edges not in
$\tilde\Omega$. We have that $\<T\cup E\>*\<T\cup
E_{\tilde\Omega}\>=A$. Further, $T\cup E_{\tilde\Omega}$ maps to
$G-p(E)$. Let $C$ denote a free factor of $\<G-p(E)\>$ complementary
to $\<T\cup E_{\tilde\Omega}\>$. There is a morphism $(T\cup E)\vee
(T\cup E_{\tilde\Omega})\vee G_C\to G$ where $G_C$ represents $C$.
\end{notation}

\begin{prop}\label{adjoined}
Assume $n=rk(\FF_n)\geq 3$.
If $G$ is a graph such that $B|G$ is embedded and a free factor $A$ is
adjoined, then $d_\F(A,B)\leq 4$.
\end{prop}

The proof is organized into several cases. One could give a quick
argument if 4 is replaced by a larger number.

\begin{proof}
Fix notation as above.

{\it Case 1: $E_{\tilde\Omega}\not=\emptyset$.}
There are two cases. If there is an embedded loop $\beta$ of
$B|G\subset G$ contained in $G- p(E)$ then we have
$$A\searrow \< T\cup E_{\tilde\Omega} \>\nearrow\< G-p(E)\>\searrow\<
\beta\>\nearrow B$$

Otherwise, choose $\beta$ in $B|G\subset G$. By our assumption,
$\beta$ is not contained in $G-p(E)$. We have
$$A\searrow \< T\cup E_{\tilde\Omega}\> \nearrow \< T \cup
E_{\tilde\Omega}\cup \beta\> =\<T\cup
E_{\tilde\Omega}\>*\<\beta\>\searrow \< \beta\> \nearrow B$$ where by
$\< T \cup E_{\tilde\Omega}\cup \beta\>$ we mean the free factor
obtained from $\< T\cup E_{\tilde\Omega}\>$ by adding an element representing the conjugacy
class of $\beta$ chosen so
that $$\<T\cup E_{\tilde\Omega}\><\<G-p(E)\><\<(G-p(E))\cup\beta\>=\<G-p(E)\>*\<\beta\>$$
is a sequence of free factors. Note that $\<T\cup E_{\tilde\Omega}\cup\beta\>$ is proper
since $rk\<T\cup E_{\tilde\Omega}\><rk A<n$.

For the rest of the proof, assume $E_{\tilde\Omega}=\emptyset$,
i.e.\ $\tilde\Omega$ is a forest and $T$ is a maximal tree for $A|G$.

{\it Case 2: $rk(G-p(E))>1$.}  Then we may choose an embedded loop
$\beta$ in $G-p(E)$ which is either contained in $B|G$ or $\beta\cup
B|G\neq G-p(E)$ and the complement $(G-p(E))-(\beta\cup B|G)$ contains
a non-separating edge. In either case, $d_\F(\< \beta\>,B)\leq
2$. Referring back to the proof of Lemma~\ref{near embedding},
$(T-E)\vee (G-p(E))\to G$ is a homotopy equivalence and both $T-E$
(representing $A$) and $\beta$ are in the complement of a
non-separating edge of $G-p(E)$. So we also have $d_\F(A,\<
\beta\>)\leq 2$.

{\it Case 3: There is only one edge in $E$.}
Then $3\le rk(G)=1+rk(G-p(E))$ and we are in Case~2.

{\it Case 4: There is $e\in E$ and an embedded loop $\beta$ in $B|G\subset G$ such that $\beta$ misses $p(e)$.}  
Then $A\nearrow \<G-p(e)\>\searrow \<\beta\>\nearrow B$.

{\it Case 5: Every embedded loop in $B|G$ meets every edge of $p(E)$.}
Then $B$ has rank 1. For $e\in E$, $(T\cup (E-e))\vee (G-p(e))\vee
B\to G$ is a homotopy equivalence (where say the base point
corresponds say to an endpoint of $e$). If $e'\not= e$ is an
edge of $E$, then $A\searrow \<T\cup e'\>\nearrow \<(T\cup e')\vee
B\>\searrow \<B\>$.
\end{proof}

\begin{example}
Consider ${\mathbb F}_4=\< a,b,c,d\>$, $B=\< a,c,d\>$, $A=\< abaab, cb,
abd\>$. Then $B$ is embedded and $A$ is adjoined in the rose
$R$. See Figure \ref{dist 4}.

\begin{figure*}[h]
\centering
\includegraphics[scale=0.5]{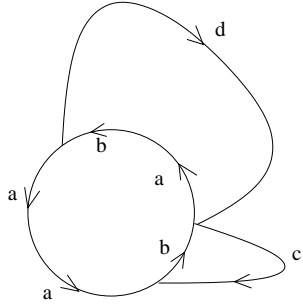}
\caption{Picture of $A|R$. The circle is $\tilde\Omega$.}
\label{dist 4}
\end{figure*}
Since $A$ and $B$ have corank 1, the distance between them is
even, and in fact it is 4 ($A$ intersects trivially all conjugates of $B$).
\end{example}

\subsection{Coloring of free factors}\label{color}

We will define a notion of a {\it color} of a free factor, so that the
projection $\pi_A(B)$ is well-defined when $A$ and $B$ have the same
color (and $A\neq B$). For example, if the color of $A$ is the image
of $H_1(A;\Z/2)\to H_1(\FF_n;\Z/2)$ then there can be no {\it embedding}
$A|G_t\to G_t$ provided $G_t$ contains a subgraph representing $B$
where $A$ and $B$ have the same color but are distinct. This notion is
not sufficient for pinching maps. For example, consider $\FF_3=\langle
a,b,c\rangle$, $B=\langle a,b\rangle$, and $A=\langle
a,cbc^{-1}\rangle$. If $G_t$ is the rose on $a,b,c$ the map $A|G_t\to
G_t$ is a pinching map, since $A|G_t$ is represented by two loops that
map to $a$ and $b$ and a connecting edge that maps to $c$.

We define the color of a free factor $A$ as the function on
$H^1(\FF_n;\Z/2)$ that to 0 assigns the $\Z/2$-homology class of $A$
(i.e. the image of $H_1(A;\Z/2)\to H_1(\FF_n;\Z/2)$) and to a nonzero
class $x$, i.e. to an index 2 subgroup $\FF_x\subset \FF$, it assigns
the $\Z/2$-homology classes of $A\cap\FF_x$ and $uAu^{-1}\cap\FF_x$ in
$\FF_x$ for $u\in \FF_n-\FF_x$. In other words, represent $A$ as a
subgraph of a graph $G$ with $\pi_1(G)=\FF_n$ and for a nontrivial double
cover $\tilde G\to G$ take the homology classes of the (one or two)
components of the preimage of
the subgraph. E.g. in the above example, $A$ and $B$ have different
colors, since unwrapping $c$ yields lifts with different homology
classes. More generally, one can show that if $A$ and $B$ are free
factors with $B$ embedded in $G$ and $A$ pinched in $G$, then either
$A=B$ or $A$ and $B$ have different colors.

Even this notion is not sufficient for near-embeddings, that is, there
are distinct free factors $A$ and $B$ that have the same color and a
graph $G$ so that $B$ is embedded and $A$ is nearly embedded in
$G$. For example, take $G$ to be the rank 3 rose with the ``identity''
marking, i.e. with loops labeled $a,b,c$, and let $B=\< c\>$, $A=\<
c[a,b]^2\>$. The commutator is squared to insure that homology classes
are equal in double covers.

\subsection{Projection is well defined}

When $A,B$ are two free factors we define the subset
$\pi_A(B)\subset\S(A)$ as the set of 1-edge splittings of $A$ obtained
as follows: choose any $G\in \X$ where $B$ is embedded and collapse
all edges but one in the graph $A|G$. We are interested in restricting
to a collection of pairs $A,B$ ``in general position'' so that
$\pi_A(B)$ is uniformly bounded.  It is easy to see that for some
pairs $A,B$ (e.g. when $A*B=\FF_n$) this set is unbounded and the
projection is not well defined.

We will argue that the projection of a free factor to the splitting
complex of another free factor is well defined assuming either:
\begin{itemize}
\item $A$ and $B$ have the same color, or
\item $A$ and $B$ are at distance $>4$ in $\F$.
\end{itemize}

First we make the following remark. Fix a graph $\Gamma_0\in\X$ so
that $B|\Gamma_0$ is a rose and $B|\Gamma_0\to\Gamma_0$ is an
embedding. If $\Gamma\in\X$ is any graph with $B|\Gamma\to\Gamma$ an
embedding, there is a difference of markings map
$\phi:\Gamma_0\to\Gamma$ that maps $B|\Gamma_0$ to $B|\Gamma$. Let
$G_t$ be the path induced by $\phi$ (see Section \ref{review}). The
first portion of the path is contained in a simplex, so it induces a
uniformly bounded path in $\S(A)$ for any free factor $A$. The second
portion is a folding path in which $B$ is embedded at all times and
for $t<t'$ $B|G_t$ maps into $B|G_{t'}$. Thus
to prove that $\pi_A(B)$ is uniformly bounded, it suffices to argue
that every folding path $G_t$ along which $B$ is embedded makes
uniformly bounded progress in $\S(A)$.

\begin{lemma}
Suppose $A$ and $B$ are distinct free factors that have the same
color. Then the projection $\pi_A(B)$ of $B$ to the splitting complex
of $A$ is a uniformly bounded set.
\end{lemma}

\begin{proof}
Let $G_t$ be a folding path along which $B$ is embedded. We will argue
that the progress of $A|G_t$ in $\S(A)$ is uniformly bounded.  By
Proposition \ref{analysis} we may further assume that $A$ is adjoined
in every $G_t$, i.e. $\tilde\Omega_t$ is a proper subgraph of $A|G_t$.

{\it Claim.} For every $t$, every point $x$ of $A|G_t-\tilde\Omega_t$ maps
to $B|G_t\subset G_t$.

Note first that it suffices to check this for points $x$ in the
interior of an edge.
If $x$ belongs to a nonseparating edge of $A|G_t$, this
follows from the assumption that $A$ and $B$ have the same image in
$\Z/2$-homology (some loop in $B$ must cross the image of $x$). If $x$
belongs to a separating edge and maps to a point $y\in G_t-B|G_t$,
consider the double cover of $G_t$ induced by $y$ and how $A|G_t$ lifts
to it. We cut $G_t$ open along $y$ and glue two copies crosswise, and
do likewise with $A|G_t$ along $x$. Visibly, $B|G_t$ lifts to two disjoint
copies, while $A|G_t$ is cut into two parts, and one part lifts to one
copy of $B|G_t$, while the other lifts to the other. This contradicts
the assumption that $A$ and $B$ have the same color. The Claim is proved.

Assuming the statement of the lemma is false, fix $t_1<t_2<t_3$ so
that $A_t$ makes big progress in $\S(A)$ on both $[t_1,t_2]$ and
$[t_2,t_3]$ and the homotopy type of $\tilde\Omega_t$ is constant for
$t\in [t_2,t_3]$ (see Proposition \ref{details}).  We now have 3
cases. Recall that it is possible that $\tilde\Omega_t$ is not a core
graph.

{\it Case 1.} At $t=t_2$ there is a non-core point
$x\in\tilde\Omega_t$ that maps to $G_t-B_t$.

Consider the preimage of $x$ in $A|G_{t_1}$. This preimage is
contained in $\tilde\Omega_{t_1}$ by the Claim. We now observe that
every edge (or even a legal path) in $\tilde\Omega_{t_1}$ can map over
$x$ at most twice. Thus the preimage of $x$ is a uniformly bounded set
and we apply Lemma \ref{distance} to deduce that the distance in
$\S(A)$ between $A|G_{t_1}$ and $A|G_{t_2}$ is bounded, contradicting
the choice of $t_1<t_2$.

{\it Case 2.} At $t=t_3$ there is a non-core point
$x\in\tilde\Omega_t$ that maps to $G_t-B|G_t$.

Then argue in the same way and deduce that the distance in $\S(A)$
between $A|G_{t_2}$ and $A|G_{t_3}$ is bounded.

{\it Case 3.} Both Case 1 and Case 2 fail, i.e. non-core points of
$\tilde\Omega_t$ map to $B|G_t$ for $t=t_2$ and $t=t_3$. 

In this case, if $\tilde \Omega_{t_3}$ maps to $B|G_{t_3}$, then all of
$A|G_{t_3}$ maps to $B|G_{t_3}$ so we must have $A\subset B$ and therefore $A=B$
(same color implies the same rank). Therefore, there must be a point
$x$ in the core of $\tilde\Omega_{t_3}$ that maps outside
$B|G_{t_3}$. The preimage of $x$ in $A|G_{t_2}$ is contained in
$\tilde\Omega_{t_2}$ by the Claim, and in fact it is contained in the
core of $\tilde\Omega_{t_2}$ by our assumption that Case 1 fails. Thus
we are in the situation of Lemma \ref{distance 2} and we deduce that
$A|G_{t_2}$ and $A|G_{t_3}$ have bounded distance in $\S(A)$,
contradicting the choice of $t_2$ and $t_3$.
\end{proof}

\begin{lemma}
Suppose that the distance in $\F$ between free factors $A$ and $B$ is
$>4$. Then the projection of $B$ to $\S(A)$ is well defined up to
a bounded set.
\end{lemma}

\begin{proof}
This follows from Propositions \ref{adjoined} and \ref{analysis},
since a folding path along which $B$ is embedded cannot have $A$
adjoined anywhere.
\end{proof}

\subsection{Verifying conditions of Theorem \ref{bbf}}

First we check the first bullet in Theorem \ref{bbf}. The key is the
construction of a ``tame path'' joining a graph in a folding path
that's making progress in $\S(A)$ to a graph where $A$ is embedded,
and so that along that path free factors far from $A$ and the factors
of the same color do not make progress.

\begin{lemma}
Assume that $n=rk(\FF_n)\geq 3$ and that a free factor $A$ is adjoined
in $G$. Then there is a path from $G$ to $G'$ so that $A$ is embedded
in $G'$ and the path makes only a uniformly bounded progress in
$\S(B)$ if $B$ is a free factor such that either $d_\F(A,B)>4$
or $A$ and $B$ have the same color.
\end{lemma}

\begin{proof}
Use the conventions of Notation~\ref{notation}. In particular, we have
a morphism $(T\cup E)\vee (T\cup E_{\tilde\Omega})\vee G_C\to
G$. Folding our morphism gives the (inverse of) the desired path
$G_t$. (If some $G_t$ is not a core graph, then we restrict to the
core and normalize so that volume is 1.)

We have to check the progress of $B|G_t$ in $\S(B)$.  If there is
$e\in E$ such that the image in $G$ of $B|G$ crosses $p(e)$ (we say
$B$ is {\it good} for $A$ if this holds) then visibly this path
doesn't make progress in $\S(B)$ since the splitting of $B$ determined
by the preimage of a point $y\in int(p(e))$ is unchanged along the
path. Note that $B$ is good for $A$ if $A$ and $B$ have the same color
(since $e$ is nonseparating in $A|G$ it suffices to consider image in
$\Z/2$-homology). $B$ is also good for $A$ if $d_\F(A,B)>4$. Indeed,
suppose $B$ does not cross $p(e)$ for any $p\in E$. There are two
subcases. If $rk(A)>1$ choose an embedded loop $\alpha$ in $A|G$
missing some $e\in E$. Then $d_\F(A,\<\alpha\>)=1$ and
$d_\F(\<\alpha\>,B)\leq 3$ since the images of $\alpha$ and $B|G$ are
in a proper subgraph of $G$. If $rk(A)=1$ choose any rank 1 factor
$\<\beta\>$ of $B$ and note that $A$ and $\<\beta\>$ lie in the same
rank 2 factor.
\end{proof}

\begin{cor}\label{adj}
Assuming free factors $A$ and $B$ have the same color or are at
distance $>4$,
in the definition of the projection
$\pi_A(B)\in\S(A)$ we may take a graph $G$ in which $B$ is adjoined,
rather than embedded.
\end{cor}

\begin{prop}
Assume $n=rk(\FF_n)\geq 3$. Then there is a constant $\xi$ such that the
following holds.
Let $A,B,C$ be 3 free distinct factors and assume that for each pair
either they have the same color or they are at distance $>4$.
Then only one of $d_A(B,C),d_B(A,C),d_C(A,B)$
can be larger than $\xi$.
\end{prop}

\begin{proof}
Assume both $d_A(B,C)$ and $d_B(A,C)$ are large. Consider a folding
path $G_t$, $t\in [\alpha,\omega]$, from a graph where $B$ is embedded
to a graph where $C$ is embedded.  By Proposition~\ref{analysis} there
is a long interval $[t_1,t_2]$ along which $A$ is adjoined and where
the path makes large progress in $\S(A)$.  Thus the portion $G_t$,
$t\in [t_2,\omega]$ following this interval is a folding path from a
graph in which $A$ is adjoined to a graph where $C$ is embedded, but
whose length is much shorter than for the original path. Now use the
fact that $d_B(A,C)$ is large plus Corollary~\ref{adj} and run the
same argument to conclude that on some long interval $[t_3,t_4]$ in
$[t_2,\omega]$ $B$ is adjoined and makes large progress in
$\S(B)$. Thus $G_t$ for $t\in [t_3,\omega]$ is a folding path from a
graph where $B$ is adjoined to a graph where $C$ is embedded, and by
Corollary \ref{quasi-geodesic} it has a bounded image in $\S(C)$. But
this contradicts the fact that on $[t_3,t_4]$ it makes large progress
in $\S(C)$.
\end{proof}

We now verify the second bullet of Theorem~\ref{bbf}.

\begin{prop}
Again assume $n\geq 3$. There is a constant $\xi'$ such that the
following holds. For any two free factors $A,B$ there are only
finitely many free factors $C$ so that 
\begin{itemize}
\item $A$ and $C$ have the same color or $d_\F(A,C)>4$,
\item $B$ and $C$ have the same color or $d_\F(B,C)>4$, and
\item $d_C(A,B)>\xi'$.
\end{itemize}
\end{prop}

We note that the first two bullets are needed so that $d_C(A,B)=\diam
\pi_C(A)\cup \pi_C(B)$ is well defined.

\begin{proof}
Fix a folding path $G_t$ between a graph where $A$ is embedded and a
graph where $B$ is embedded. If $d_C(A,B)$ is large, there is a middle
interval of length $\geq 1$ (say) along which $C|G_t$ has bounded
volume. The number of immersions to a fixed graph $\Gamma$ with bounded volume
is bounded in terms of the injectivity radius of $\Gamma$. It follows
that only a bounded number of the middle intervals for different $C$'s
can contain the same point, so the number of these intervals is
finite.
\end{proof}

\section{Consequences}

\begin{thm}\label{product}
$Out(\FF_n)$ acts on a finite product $Y=Y_1\times\cdots\times Y_k$ of
  hyperbolic spaces so that every exponentially growing automorphism
  has positive translation length.
\end{thm}

In the proof we need the following theorem, announced by
Handel-Mosher \cite{S}. Recall \cite{tits1} that every $EG$ stratum of a $\Phi\in
Out(\FF_n)$ has an associated {\it attracting lamination} $\Lambda$,
carried by a unique free factor $Supp(\Lambda)$ of $\FF_n$. The
lamination $\Lambda$ has uncountably many recurrent leaves and none of
them are carried by a proper free factor of $Supp(\Lambda)$.

\begin{thm}
[Handel-Mosher]
\label{hm:translate}
Suppose $\Phi$ has an $EG$ stratum whose associated lamination $\Lambda$
has full support, i.e. $Supp(\Lambda)=\FF_n$. Then $\Phi$ acts with
positive translation length on the splitting complex $\S$.
\end{thm}

\begin{proof}[Proof of Theorem \ref{product}]
Let $Y_i$ be the hyperbolic space produced by Theorem \ref{bbf} for
the collection of free factors colored by color $i$. Then $Out(\FF_n)$
acts on the product $Y_1\times\cdots\times Y_k$.

Let $\Phi\in Out(\FF_n)$ be exponentially growing. Choose an EG stratum
and let $\Lambda$ be the associated attracting lamination. There is a
smallest free factor $A$ that carries $\Lambda$ and $A$ is
$\Phi$-periodic \cite{tits1}. Thus for some $j>0$ we have
$\Phi^j(A)=A$, so the restriction of $\Phi^j$ to $\S(A)$ has positive
translation length, by Theorem~\ref{hm:translate}. It follows that $\Phi^j$
preserves the copy of $\S(A)$ in the suitable coordinate $Y_i$ and has
positive translation length.
\end{proof}

Recall that we have a coarse projection $\pi_A$ defined on the
complement of the 4-ball around $A\in\F$ with values in $\S(A)$. We
claim $\pi_A$ is uniformly coarsely Lipschitz. Indeed, if $B_1<B_2$
choose a graph $G\in\X$ so that both $B_1$ and $B_2$ are embedded in
$G$. We conclude that $\pi_A(B_1)\cap \pi_A(B_2)\neq\emptyset$, which
proves our claim.

\begin{thm}[Bounded geodesic image theorem]
Let $(A_i)$ be a geodesic in $\F$ and $A\in\F$ at distance $>4$
from each $A_i$. Then the set of projections of the $A_i$'s to $\S(A)$
is a uniformly bounded set.
\end{thm}

\begin{proof} It suffices to prove the theorem for finite
  geodesics. We will use the fact that folding paths in $\X$ project
  to uniform reparametrized quasi-geodesics in $\S$ (see the
  appendix). There is a folding path $G_t$ whose image in $\F$ is a
  uniform distance, say $K$, from $(A_i)$. Since the projection (where
  defined) from $\F$ to $\S(A)$ is Lipschitz, we may cut out a portion
  of $(A_i)$ at distance $\leq K+4$ from $A$ and argue that the
  projections of the remaining (one or two) geodesics are
  bounded. Again using that the projection is Lipschitz we may replace
  the $A_i$'s by the image of (one or two) folding paths. But a
  folding path doesn't make progress in $\S(A)$ unless it projects (in
  $\F$) close to $A$, by Propositions~\ref{analysis} and
  \ref{adjoined}.
\end{proof}

\section{Final remarks}\label{further}

\subsection{Failure of quasi-isometric embedding}

We sketch two examples showing that orbit maps $Out(\FF_n)\to Y$ in Theorem
\ref{product} are not quasi-isometric embeddings.

First, we claim that any polynomially growing automorphism $\Phi$ has
bounded orbits in $Y$. By \cite{BBF} this is equivalent to the
assertion that a folding path from $\Gamma$ to $\Gamma\Phi^m$ makes
uniformly bounded progress in every $\S(A)$, where $\Gamma$ is a
relative train track representative for $\Phi$ (see \cite{BH}). A
folding path $G_t$ can be chosen to preserve the stratification. Now
consider $A|G_t$ on the middle segment. This segment can be subdivided
into a bounded number of subsegments on which the image of $A|G_t\to
G_t$ has constant rank. On each subsegment the splitting of $A$
induced by an edge of $A|G_t$ that maps to the highest stratum of the
image is unchanged along the whole subsegment, and this proves the
claim. 

The second example, pointed out to us by Ursula Hamenst\"adt, consists
of a sequence of automorphisms that are not powers of one
automorphism. Take $n=3$ and $\phi_i:\<a,b,c\>\to \<a,b,c\>$ given by
$a\mapsto a$, $b\mapsto b$, $c\mapsto cw_i$, where $w_i$ is an
arbitrary (complicated) element of $\<a,b\>$. If the $w_i$'s get
longer, the images $\Phi_i$ of $\phi_i$ in $Out({\mathbb F}_3)$ go to
infinity, but an argument similar to the above shows that the natural
folding path from the ``identity rose'' $R$ to $R\Phi_i$ does not make
progress in any $\S(A)$.

\subsection{The case of $\F$}
One may replace the splitting complexes $\S(A)$ by the complexes of
free factors $\F(A)$ throughout the paper. A folding path $G_t$ can
make progress in $\F(A)$ only on the middle interval in Lemma \ref{3
  intervals}. In addition, now one can assert that there is only
bounded progress when $\tilde\Omega_t$ is not a forest. Rephrasing
Lemma \ref{uniform crossing}, if $A$ is malnormal there is a
subinterval of the middle interval where $A$ is nearly embedded and
outside this interval progress in $\F(A)$ is uniformly bounded. Using
Lemma \ref{near embedding} we see that when $A$ is malnormal but not a
free factor, the projection $\pi^\F_A(B)\subset\F(A)$ is defined for
every free factor $B$ and all these sets are contained in a uniformly
bounded subset of $\F(A)$, in other words $A$ ``comes'' with a
free factor, canonical up to a bounded distance in $\F(A)$ (assuming
$rk(A)>1$ and in case $rk(A)=2$ we modify the definition of $\F(A)$ as
usual so that it is the Farey graph).

Returning to the case of free factors, Main Theorem holds with
projections in $\F(A)$ without any change. We thus get an isometric
action of $Out(\FF_n)$ on a product of hyperbolic spaces. This time, an
automorphism has positive translation length if it has a {\it bottom}
EG stratum. 

Likewise, the Bounded geodesic image theorem holds without change
after replacing $\S(A)$ with $\F(A)$.

\appendix
\renewcommand\thesection{Appendix \Alph{section}}

\section{Folding paths and hyperbolicity of $\S$}
The goal of this appendix is to revisit the Handel-Mosher proof of the
hyperbolicity of $\S$ using folding paths. We prove a technical lemma
that allows us to use (the usual) folding paths with $\geq 2$ gates at
every vertex and this simplifies some of the details of the proof. We
will work with {\it liberal folding paths} that include the original
Stallings paths \cite{stallings}, the {\it greedy} folding paths as
used in this paper and in \cite{F}, and the Handel-Mosher folding
paths. The consequence is that liberal folding paths project to
reparametrized quasi-geodesics in $\S$. Our proof follows closely the
strategy in \cite{S}, but it is relatively short partially because we
don't need to be arranging 3 gates condition, and partially because we
don't provide every detail. We also note that Hilion-Horbez
\cite{HH} give a proof of the hyperbolicity of $\S$ using the strategy
of \cite{S}, but using Hatcher's surgery paths in the sphere complex
and also find simplifications from this point of view.

\renewcommand\thesection{\Alph{section}}

\subsection{Liberal folding paths}
In this appendix, we work with the barycentric subdivision $\S'$ of
$\S$. A vertex of $\S'$ can be thought of as a minimal nontrivial
simplicial $\FF_n$-tree $T$ with trivial edge stabilizers, or
alternatively as the Bass-Serre orbit space $G=T/\FF_n$, and we will freely
switch between these two points of view. A $k$-simplex of $\S'$ is a
chain $T_0\to T_1\to\dots\to T_k$ of maps between vertices of $\S'$
where each $T_i\to T_{i+1}$ is obtained by equivariantly collapsing
some non-empty set of edges. If $\XS$ denotes the space of metric,
simplicial $\FF_n$-trees with trivial edge stabilizers (up to
equivariant isometry) then there is a surjective forgetful map
$\XS\to{\S'}^{(0)}\subset\S'$. If $T\mapsto S$, we say $T$ {\it represents} $S$.

The {\it natural $CW$-structure or triangulation on $G=T/\FF_n$} is the
one induced by the natural $CW$-structure on $T$. We follow \cite{S}
and, when working with a specific (probably not natural) triangulation
refer, to edges as {\it edgelets}.

Recall that a {\it morphism} between two metric trees is an
equivariant map which is an isometric embedding on each edgelet for
some triangulation of the source. Note that every map between trees
that embeds every edge becomes a morphism after choosing a metric in
the target and pulling it back. It is always possible to choose a
metric so that after suitable subdivisions the map is simplicial and
all edges have the same length. Thus our category is morally
equivalent to the Stallings category \cite{stallings}.

By an {\it optimal} morphism we mean a morphism with $\geq 2$ gates at
every vertex. If $T,T'\in\XS$ so that every element of $\FF_n$ which is
elliptic in $T$ is also elliptic in $T'$, then there is an optimal
morphism $T''\to T'$ where $T''$ is obtained from $T'$ by
equivariantly collapsing some edges and possibly changing the metric.

We allow {\it liberal folding}, i.e.\ illegal turns can be folded with
any speed, folding can pause and then continue, and we even allow not
doing anything on a subinterval. More precisely, a {\it (liberal)
  folding path guided by optimal $f$} is
\begin{itemize}
\item
a path $G_t$, $t\in [0,L]$, in $\XS$ and 
a collection $\{G_{t_1}\to G_{t_2}\}_{0\le t_1<t_2\le L}$ of morphisms
\end{itemize}
such that 
\begin{itemize}
\item
$G_0\to G_L$ is $f$ and
$G_{t_1}\to G_{t_2}\to G_{t_3}$ is a factorization of $G_{t_1}\to
G_{t_3}$ for all $0\le t_
1<t_2<t_3\le L$.
\end{itemize}
We occasionally use $[G_{t_1},G_{t_2}]$ to denote $\{G_t\mid t\in
[t_1,t_2]\}$.  The proof of \cite[Proposition~2.2]{F} shows {\it
  mutatis mutandis} that, given $f$, there is a
(greedy) folding path guided by $f$.

\subsection{Collapses and expansions}
We will denote trees by $\tilde G$ (with various sub and superscripts)
and the quotient graphs with $G$ (with the same sub and
superscripts). If $G_t$, $t\in [0,L]$, is a liberal folding path and
if $F_L\subset \tilde G_L$ is a proper equivariant forest, define
$\tilde G_t'=\tilde G_t/F_t$ where $F_t$ is the preimage of $F_L$ in
$\tilde G_t$.

$$
\begindc{1}[5]
\obj(6,12)[12]{$G_0$}[\south]
\obj(6,18)[13]{$G'_0$}
\obj(18,12)[32]{$G_L$}[\south]
\obj(18,18)[33]{$G'_L$}
\mor{12}{32}{}[\atright,\solidarrow]
\mor{12}{13}{}
\mor{12}{32}{}
\mor{13}{33}{}
\mor{32}{33}{}
\enddc
$$

\begin{lemma}\label{A1}
$G_t'$ is a liberal folding path.
\end{lemma}

Note that even if $G_t$ is a greedy folding path, the quotient $G_t'$
may not be greedy since we could collapse the parts where folding
takes place (so in fact $G_t'$ may have subintervals on which it is
constant). 

\begin{proof}
There is a natural map $f':\tilde G_0'\to \tilde G_L'$ with respect to
which $\tilde G_t'$ is a liberal folding path guided by $f'$, but we must
show that every vertex has at least 2 gates. See \cite[Proposition
  4.4]{S}. This is clear except possibly at vertices coming from the
components of $F_0$. If $d,d'$ are distinct directions coming out of
one such component $C$ and if $d'=g(d)$ for some $g\in Stab(C)$, then
$d,d'$ must map to distinct directions out of $f(C)$ since direction
stabilizers are trivial. Thus we are reduced to considering the case
when $Stab(C)=1$, i.e.\ when $C$ is a nondegenerate finite tree. If all
directions out of $C$ map to the same direction out of $f(C)$, say
based at $v\in f(C)$, then the finite nondegenerate tree $f(C)$ has an
endpoint $w$ distinct from $v$. Any vertex of $C$ that maps to $w$ can
have only one gate, contradiction.
\end{proof}

There is a reverse operation of {\it expansion}. Let $G_t$, $t\in
[0,L]$, be a folding path and let $\tilde G_L'\to \tilde G_L$ be a
collapse map. Define $\tilde G_t'$ as the pull-back of the maps
$\tilde G_t\to \tilde G_L$ and $\tilde G_L'\to \tilde G_L$. The
induced map $\tilde G_t'\to \tilde G_L'$ is an isometry on subtrees
that map to point in $\tilde G_t$.

$$
\begindc{1}[5]
\obj(6,12)[12]{$G_0$}[\south]
\obj(6,18)[13]{$G_0'$}
\obj(18,12)[32]{$G_L$}[\south]
\obj(18,18)[33]{$G'_L$}
\mor{12}{32}{}[\atright,\solidarrow]
\mor{13}{12}{}
\mor{12}{32}{}
\mor{13}{33}{}
\mor{33}{32}{}
\enddc
$$

As $t$ increases the path $G_t'$ is a liberal folding path except at
times $t$ when in the path $G_t$ distinct preimages of points in $G_L$
that got blown up get identified under folding. At those times
different copies of blow-up trees get identified. To transform $G_t'$
into a folding path insert an interval where these different copies
get identified (equivariantly) by folding. Thus the blow-up operation is not
completely canonical as it requires a choice (the choice could be made
canonical by taking the greedy identification path).

\begin{lemma}\label{pullback}\label{A2}
Thus modified, $G_t'$ is a liberal folding path.\qed
\end{lemma}

Thus given a folding path and an expansion of the terminal graph,
there is an induced folding path. We will normally suppress the
required insertions of constant paths and reparametrizations.

\subsection{Distance estimates in $\S'$}
With
little change, the proof of Lemma~\ref{distance} shows:
\begin{lemma}\label{simple}
Let $f:G\to G'$ be an optimal morphism. Then $d_{\S'}(G,G')$ is
bounded by a linear function of the cardinality of $f^{-1}(y)$ for any
$y\in G'$.

In particular, if $G_t$ is a folding path guided by $f$ then the
diameter of the image of the folding path in $\S'$ is bounded by a
linear function of the cardinality of $f^{-1}(y)$ for any $y\in G'$.
\end{lemma}

In the next lemma we will consider finite sets $R$ and $B$ of red and
blue points in $G'$. We will assume $R\cap B=\emptyset$ and also that
$R\cup B$ is disjoint from the vertex set. A {\it mixed region} in
$G'$ is a complementary component of $R\cup B$ whose frontier
intersects both $R$ and $B$. By pulling back red and blue points, we
obtain red and blue points in $G$, and we can talk about mixed regions
in $G$. The following lemma is a restatement of \cite[Proposition
  6.2]{HH} in the setting of optimal morphisms.

\begin{lemma}\label{complicated}
Let $f:G\to G'$ be an optimal morphism. Suppose $R$ and $B$ are
nonempty disjoint finite sets in $G'$ also disjoint from the set of
natural vertices. Then $d_{\S'}(G,G')$ is bounded by a linear function
of the number of mixed regions in $G$.
\end{lemma}

This generalizes Lemma \ref{simple}, e.g.\ take $R=\{y\}$ and $B$ is a
single nearby point. In practice, $R$ and $B$ will be chosen so that
$R\cup B$ intersects every edge in one point, and then the
complementary regions are trees, and in this form the lemma appears in
\cite{S}. Note that bounding the number of mixed regions in $G$
amounts to saying that an edge of $G$ can map over both colors many
times, but the number of times it goes from one color to the other is
bounded.

\begin{proof}
We subdivide $G'$ so that the restriction of $f$ to each edge
(i.e.\ edgelet) of the induced subdivision of $G$ is an isometry and
each edgelet of the subdivision of $G'$ contains at most one colored
point. Thus each edgelet is colored blue, red, or neutral. By
collapsing all neutral edgelets, we may assume there are only red and
blue edgelets. In particular, mixed regions correspond to {\it mixed
  vertices}, i.e.\ vertices adjacent to both red and blue
edgelets. Consider a folding path $G_t$, $t\in [0,L]$, guided by
$f$. We impose the requirement that the folding path is obtained by
iteratively folding only red edgelets for as long as possible; and
then only folding blue edgelets for as long as possible, etc. Thus the
folding path is broken up into segments, called {\it phases}, and we
talk about red and blue phases, depending on which color is allowed to
get identified. The proof now follows from three observations:

\begin{itemize}
\item 
Under folding the number of mixed vertices never goes up; in fact the
complexity $(N,-r)$ can never increase, where $N$ is the number of
mixed vertices and $r$ is the sum of the ranks of the stabilizers over
the orbits of mixed vertices (viewing $G_t$ now as trees). More
precisely, the complexity changes only in the event that two mixed
vertices merge into one -- if they are in different orbits $N$
decreases, and if they are in the same orbit $r$ increases.
\item 
The progress in $\S'$ is bounded during any phase. E.g.\ during a
  red phase, blue edgelets define splittings that don't change during
  this phase.
\item 
Complexity has to decrease at least once during each phase, except
possibly first and last.
\end{itemize}

To prove the third bullet, consider a phase, say red, starting with
$G_t$. Viewing $G_t$ as a tree, it is decomposed into maximal red and
blue subtrees. If our red phase is not the first phase, each component
of the blue
subforest of $G_t$ embeds in $G_L$. If our red phase is not the last,
then there is a first time during the phase when a red fold identifies
points in distinct blue subtrees and hence complexity decreases.
\end{proof}

\begin{definition}
A {\it hanging tree} is a tree with a base vertex $v$ and a train
track structure in which $v$ and all valence one vertices have only
one gate, and every other vertex has two gates, with the direction
towards $v$ being its own gate. We draw a hanging tree with $v$ on top
and call it the top vertex.
\end{definition}

\begin{example}
Consider the map $f:R\to R$ where $R$ is the rank 3 rose with the
identity marking given by $x\mapsto x$, $y\mapsto x^kyx^{-k}$ and
$z\mapsto x^kzx^{-k}$. Then $f$ can be thought of as obtained from the
identity $R\to R$ by ``rotating'' $k$ times around $x$. The preimage
of a point in $x$ is contained in the loop $x$ together with the
hanging tree consisting of initial and terminal segments in edges
labeled $y$ and $z$. This is an example where (i) fails in the next
lemma. 
\end{example}

\begin{lemma}\label{technical}
Let $f:G\to G'$ be an optimal morphism between graphs in Outer space
and suppose that $e\subset G$ and $e'\subset G'$ are edges that define
the same splitting. Then one of the following holds:
\begin{enumerate}[(i)]
\item The cardinality of the preimage of a point $y'\in \int e'$ is
  uniformly bounded in $G$.
\item There is a hanging tree $H\subset G$ with one endpoint of $e$
  attached to the top vertex and the other endpoint to a point of $H$
  so that $e$ together with a (possibly degenerate) segment of $H$
  forms an embedded legal loop $E$. The preimage of some $y'\in e'$ is
  contained in $H\cup e$ and intersects $E$ in one point. Both
  endpoints of $e$ have two gates.
\end{enumerate}
\end{lemma}

\begin{proof}
Construct a folding path $G_t$ guided by $f$ so that on an initial
segment from $G$ to say $G_s$ the image of $e$ is not involved in any
folds and so that the map $G_s\to G'$ does not admit any folds not
involving the image $e_s$ of $e$. Let $\hat e_s$ be the natural edge
in $G_s$ that contains the edge $e_s$.  By assumption, there are no
illegal turns in $G_s$ except for those involving $e_s$. Now consider
cases.

{\it Case 1. $G_s-\hat e_s$ has no valence 1 vertices.}
Then $G_s-\hat e_s\to G'$ is an immersion and on fundamental groups induces
an isomorphism $\pi_1(G_s-\hat e_s)\to \pi_1(G'-e')$ (if $\hat e_s$ and $e'$ are
separating, we mean an isomorphism for each component). It follows that
$G_s-\hat e_s\to G'-e'$ is an embedding (in fact, a homeomorphism, unless
$e'$ is a loop connected to the rest of the graph by a separating
edge).

Now since $G_s\to G'$ is a homotopy equivalence, the image of $\hat
e_s$ crosses $e'$ exactly once. Thus for any point $y'$ in $e'$
the preimage in $G_s$ is a single point $y_s$ in $\hat e_s$. Now
note that if $\sigma$ is any edge (or a legal path) in $G-e$, its
image in $G_s$ is disjoint from $e_s$, so it can cross $y_s$ at most
twice, and conclusion (i) follows.

Whenever the endpoints of $\hat e_s$ are distinct, we are in Case~1. Moreover,
we may assume that $\hat e_s$ is a loop with vertex $a$ that has valence 3.

{\it Case 2. $\hat e_s$ is a loop with vertex $a$ with valence 3, and
  the two directions into $\hat e_s$ are equivalent.} Then the turn
between the two $\hat e_s$-directions is the only illegal turn in
$G_s$, so $G'$ is obtained from $G_s$ by folding
the ``monogon'' $\hat e_s=e_s$. Now take $y'\in G'$ to be a point in the image
of $e_s$. The preimage of $y'$ has at most two points in $G_s$ and is
contained in $e_s$. The image in $G_s$ of an edge $u\neq e$ in $G$
cannot cross $e$ so (i) again holds.

{\it Case 3. All turns in $G_s$ are legal.} Then $G_s=G'$ so (i) holds
with $y'$ a point in $e_s$.

{\it Case 4. $\hat e_s$ is a legal loop with vertex $a$ with valence 3
  and two gates.}  Thus there is a separating edge $z$ incident to
$a$. Then the complement of $\hat e_s\cup z$ embeds into $G'-e'$ as in
Case~1. Now fix $y'\in e'$, and note that the preimage of $y'$ in
$G_s$ is contained in $\hat e_s\cup z$ and meets $\hat e_s$ in a
single point.  We claim that the preimage of $\hat e_s\cup z$ in $G$
is the union of $e$ and a hanging tree, with one endpoint of $e$
attached to the top of the tree, and the other endpoint somewhere on
the tree, and so that $e$ is contained in a legal embedded loop. Once
established, the claim proves that (ii) holds.

One end of $e_s\subset G_s$ must be involved in the illegal
turn. Without loss, say $e_s$ is oriented so that the terminal
direction is involved in the illegal turn. If $u$ is any point of
$z\cup \hat e_s-e$ there is a unique embedded legal path $P_u$
starting at $u$ and ending at the initial vertex of $e_s$.  If $x\in
G$ is any point that maps to some $u\in z\cup \hat e_s-e$ (except the
endpoint of $z$ not in $e_s$), using the fact that $G\to G'$ is
optimal we see that there is a unique legal path $Q_x$ in $G$ that
maps isometrically to $P_u$ and ends at the initial vertex of $e_s$
(one can lift $P_u$ using that there are $\geq 2$ gates everywhere;
the lift must end at the initial vertex of $e$ since $G\to G_s$ is 1-1
over $e_s$; the lift is unique since otherwise a loop gets killed). The
union of all $Q_x$'s is a hanging tree with the top vertex at the
initial vertex of $e$.
\end{proof}

\begin{remark}\label{more}
If $f:G\to G'$ is an optimal map and $f(x_1)=f(x_2)$ then $x_1$ and
$x_2$ are joined by a unique immersed path whose image is
nullhomotopic rel endpoints. The existence follows from
$\pi_1$-surjectivity of $f$, and uniqueness from
$\pi_1$-injectivity. We call this path a {\it vanishing path}. Now
suppose that we are in the situation of Lemma \ref{technical} and that
(ii) holds. If $x_1,x_2\in f^{-1}(y')$ then the vanishing path between
$x_1$ and $x_2$ is contained in $H\cup e$. This can be seen by
considering the factorization $G\to G_s\to G'$ from the proof of Lemma
\ref{technical}. To elaborate, vanishing paths for points in the
preimage of $y'$ in $\hat e_s\cup z\subset G_s$ are contained in $\hat
e_s\cup z$ since the loop $\hat e_s$ folds with (a part of) $z$ on its
way to $G_L$ and identifies all points in the preimage of $y'$. Since
the map $H\cup e\to \hat e_s\cup z$ is a $\pi_1$-isomorphism, the same
argument as above shows that there is a path $\alpha$ in $H\cup E$ connecting
$x_1$ and $x_2$ whose image (tightened) is the vanishing path between
the images of $x_i$, hence $\alpha$
is vanishing as well.

Note that every immersed path in $H\cup e$ has at most one illegal turn, hence
a vanishing path has exactly one illegal turn. If $x_1,x_2\in
f^{-1}(y')-E$ are connected by a legal segment $J$, say oriented
towards $E$, then the vanishing
path from $x_1$ to $x_2$ necessarily has the form $Ju\cdot E^{-k}u^{-1}$
where $\cdot$ denotes the illegal turn. See Figure \ref{f:more}. 

\begin{figure}[h]
\begin{center}
\input{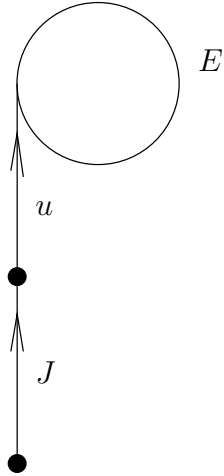}
\caption{Vanishing path in $\hat e_s\cup z\subset G_s$}
\label{f:more}
\end{center}
\end{figure}

In particular, given any folding path $G_t$ guided by $f$ (not
necessarily passing through $G_s$ above), if $x_1$ and $x_2$ are
identified in some $G_t$, then they are both identified with $y\in
f^{-1}(y')\cap E$ (the vanishing path must completely
fold). Consequently, if $x_1,\cdots,x_k\in f^{-1}(y')$ are identified
in $G_t$, then either $k$ is bounded, or all these points are
identified with $y$ (the hanging tree $H$ has bounded complexity and
given a sufficiently large subset, two of the points are connected by
a legal path).
\end{remark}

\begin{prop}\label{technical2}
Suppose $f:G\to G'$ is an optimal morphism such that $G$ and $G'$ have
edges that determine the same splitting. Then the image in $\S'$ of a
liberal folding path $G_t$, $t\in [0,L]$, guided by $f$ is bounded.
\end{prop}

\begin{proof}
By blowing up $G_L$ if necessary and applying
Lemma~\ref{pullback}, we may assume that $G_L$
is in Outer space, i.e.\ point stabilizers are trivial. 
Apply Lemma \ref{technical} to $G_0\to G_L$. If (i) holds, then the
conclusion follows from Lemma~\ref{simple}, so suppose (ii)
holds. Let $t_1$ be the first time that the preimage of $y'$ in
$G_{t_1}$ consists
of a single point. Then for any small $\epsilon>0$ the paths $G_t$ for
$t\in [0,t_1-\epsilon]$, $t\in [t_1-\epsilon,t_1]$ and $t\in [t_1,L]$
map to bounded sets in $\S'$. Indeed, this is true for the last one by
Lemma \ref{simple}, for the middle path by the smallness of
$\epsilon$, and for the first path again by Lemma \ref{simple} applied
to any preimage of $y'$ in $G_{t_1-\epsilon}$ not in the image of
$E$ (by the last sentence in Remark \ref{more} the preimage in $G_0$
of such a point is bounded).
\end{proof}

\subsection{Combing diagrams and rectangles}
A {\it combing rectangle} is a pair of folding paths $G_t$ and $G_t'$
together with collapse maps $G_t\to G_t'$ that commute with folding
maps so that the top row is obtained from the bottom row either by
collapsing as in Lemma \ref{A1} or by expanding as in Lemma
\ref{A2}. We represent a folding rectangle by a rectangle with sides
oriented to indicate the direction of folding and of collapse. The
folding direction is horizontal and the collapsing is vertical. A
combing diagram consists of combing rectangles and folding paths. All
horizontal rows are (liberal) folding paths and in particular all
horizontal maps are optimal.

A {\it zig-zag path} is an alternating sequence of collapses and
expansions. Every geodesic path in $\S'$ is a zig-zag path. 
We don't insists here on minimal triangulations where all
vertices have valence $>2$. 

\subsection{Simplifying}
Here we introduce simplifying moves on combing diagrams. 
Suppose we have a combing diagram
$$
\begindc{1}[5]
\obj(6,12)[12]{$G_0$}[\south]
\obj(6,18)[13]{$G_0^{''}$}[\west]
\obj(18,12)[32]{$G_L$}[\south]
\obj(18,18)[33]{$G^{''}_L$}[\east]
\obj(6,24)[15]{$G_0'$}
\obj(18,24)[35]{$G'_L$}
\mor{12}{32}{}[\atright,\solidarrow]
\mor{13}{12}{}
\mor{12}{32}{}
\mor{13}{33}{}
\mor{33}{32}{}
\mor{15}{35}{}
\mor{33}{35}{}
\mor{13}{15}{}
\enddc
$$
Let $F, F'$ be the forests in $\tilde
G_L''$ collapsed by the two maps. 
Then
\begin{itemize}
\item
If $F^{''}=F\cap F'$ contains a nondegenerate segment, there is a combing
diagram
$$
\begindc{1}[5]
\obj(6,12)[12]{$G_0$}[\south]
\obj(6,18)[13]{$G_0^{'''}$}[\west]
\obj(18,12)[32]{$G_L$}[\south]
\obj(18,18)[33]{$G^{'''}_L$}[\east]
\obj(6,24)[15]{$G_0'$}
\obj(18,24)[35]{$G'_L$}
\mor{12}{32}{}[\atright,\solidarrow]
\mor{13}{12}{}
\mor{12}{32}{}
\mor{13}{33}{}
\mor{33}{32}{}
\mor{15}{35}{}
\mor{33}{35}{}
\mor{13}{15}{}
\enddc
$$
where $\tilde G_t'''$ is obtained from $\tilde G_t''$ by collapsing
the components of the preimages of the forest $F''$.
\item If $F\cup F'\neq \tilde G_L''$, there is a combing diagram
$$
\begindc{1}[5]
\obj(6,12)[12]{$G_0$}[\south]
\obj(6,18)[13]{$G_0^{'''}$}[\west]
\obj(18,12)[32]{$G_L$}[\south]
\obj(18,18)[33]{$G^{'''}_L$}[\east]
\obj(6,24)[15]{$G_0'$}
\obj(18,24)[35]{$G'_L$}
\mor{12}{32}{}[\atright,\solidarrow]
\mor{12}{13}{}
\mor{13}{33}{}
\mor{32}{33}{}
\mor{15}{35}{}
\mor{35}{33}{}
\mor{15}{13}{}
\enddc
$$
where $\tilde G_t'''$ is obtained from $\tilde G_t''$ by collapsing
the preimage of $F\cup F'$.
\end{itemize}

The following is our version of Handel-Mosher's Proposition~6.5 from \cite{S}.

\begin{thm}[Handel-Mosher]\label{hm}
For every combing diagram
$$
\begindc{1}[30]
\obj(7,3)[73]{$\Gamma$}
\obj(2,2)[12]{$G_0^{}$}[\west]
\obj(2,1)[21]{$G_0^{''}$}[\south]
\obj(2,3)[23]{$G_0^{'}$}
\obj(5,1)[51]{$G_L^{''}$}[\south]
\obj(5,2)[42]{$G_L^{}$}[\east]
\obj(5,3)[53]{$G_L^{'}$}
\mor{12}{21}{}
\mor{12}{23}{}
\mor{42}{51}{}
\mor{42}{53}{}
\mor{12}{42}{}
\mor{21}{51}{}
\mor{23}{73}{}
\enddc
$$
there is $0\le K\le L$ and a
combing diagram
$$
\begindc{1}[30]
\obj(7,3)[73]{$G_L^{'}$}
\obj(9,3)[93]{$\Gamma$}
\obj(2,1)[21]{$G_0^{''}$}[\south]
\obj(2,3)[23]{$\hat G_0^{'}$}
\obj(5,1)[51]{$G_K^{''}$}[\south]
\obj(5,2)[62]{$G_K^{'''}$}[\east]
\obj(5,3)[53]{$\hat G_K^{'}$}
\obj(7,1)[71]{$G_L^{''}$}[\south]
\obj(2,2)[32]{$G_0^{'''}$}[\west]
\mor{21}{32}{}
\mor{23}{32}{}
\mor{32}{62}{}
\mor{21}{71}{}
\mor{23}{93}{}
\mor{51}{62}{}
\mor{53}{62}{}
\enddc
$$ so that the folding paths $[\hat G_K', G_L']$ and $[G_K^{''}, G^{''}_L]$
are uniformly bounded in $\S'$.
\end{thm}

\begin{proof}
Equivariantly subdivide $\tilde G_L^{}$ into finitely many orbits of edgelets so
that each edgelet either collapses or maps isometrically by either
vertical map. Using the simplifying operations from section A.5 and
replacing the middle folding path by a quotient, we may assume that
every edgelet in $G_L$ collapses under exactly one of the two vertical
maps. We will color edgelets in red and blue so that red edgelets
collapse under $G_L\to G_L'$ and blue edgelets under $G_L\to
G_L''$. By further subdividing, we will also assume that 
\begin{itemize}
\item
the preimage of the interior of an edgelet in $\tilde G_0^{}$ contains
no natural vertices,
\item the restriction of $\tilde G_0^{}\to \tilde
G_L^{}$ to each component of the preimage of an edgelet is an
isometry, and
\item blue edgelets of $\tilde G_L$ map isometrically to
  $\tilde\Gamma$ via $\tilde G_L\to \tilde G_L'\to\tilde\Gamma$.
\end{itemize}

For $t\in [0,L]$, color $G_t^{}$ by pulling back via $G_t^{}\to
G_L^{}$. In particular, $G_t^{}$ is colored red and blue, $G_t\to
G_t^{''}$ is obtained by collapsing components of blue
subforests (so $G^{''}_t$ is colored red), and $G_t^{}\to
G'_t$ is obtained by collapsing red subforests (so $G'_t$
is colored blue). To focus on the colors, we write, for
example, $G_t^{R}$ for $G_t$ with all components of the
blue subforests collapsed. So, $G_t^{R}=G_t^{''}$ and is colored
red. Similarly, $G_t^{B}=G_t^{'}$ is colored
blue.

  By Lemma~\ref{complicated} applied to blue and red, there is $K\in
  [0,L]$ so that $[G_K,G_L]$ is bounded in $\S'$ and a natural edge $e$
  of $G_K$ that contains a red arc with at least one blue arc on
  either side. By further subdividing $G_L$, we may also assume:
\begin{itemize}
\item
the preimage in $G_K$ of the interior of an edgelet of $G_L$ contains
no vertices. 
\end{itemize}
In particular, $G_K$ is a union of colored edgelets with pairwise
disjoint interiors each mapping isometrically to an edgelet of
$G_L$. Let $r$ be a red edgelet in $e$ with at least one blue edgelet on
each side in $e$ and so that there is a blue edgelet $b$ in $e$
immediately adjacent to $r$.

  Denote by $G^{rB}_K$ the result of collapsing in $G_K$ 
all red edgelets except $r$ and
  by $G^{r}_K$ the result of additionally collapsing all
  blue edgelets. For $t\in [0,K)$, $G^{rB}_t$ and $G^{r}_t$ are obtained
    as above by collapsing components of preimages of edgelets. We have the
    combing diagram:
$$
\begindc{1}[30]
\obj(10,4)[104]{$\Gamma$}
\obj(3,1)[21]{$G_0^{R}$}[\south]
\obj(3,3)[23]{$G_0^{rB}$}[\west]
\obj(6,1)[51]{$G_K^{R}$}[\south]
\obj(6,2)[62]{$G_K^{r}$}[\east]
\obj(6,3)[53]{$G_K^{rB}$}[\east]
\obj(3,4)[34]{$G_0^{B}$}
\obj(6,4)[64]{$G_K^{B}$}
\obj(8,4)[84]{$G_L^{B}$}
\obj(8,1)[71]{$G_L^{R}$}[\south]
\obj(3,2)[32]{$G_0^{r}$}[\west]
\mor{21}{32}{}
\mor{23}{32}{}
\mor{32}{62}{}
\mor{21}{71}{}
\mor{23}{53}{}
\mor{23}{53}{}
\mor{53}{64}{}
\mor{23}{34}{}
\mor{34}{104}{}
\mor{51}{62}{}
\mor{53}{62}{}
\enddc
$$

By assumption, there is an equivariant homeomorphism $h:G^{rB}_K\to
G^{B}_K$ preserving edgelets not in the orbit of $r\cup b$ and mapping
$r\cup b$ to $b$. Roughly, our desired diagram is obtained by deleting
the top left element $G^{B}_0$ from the diagram above. To make this
precise, we have to adjust the metric on $G_t^{rB}$ and replace
$G^{rB}_K\to G^{B}_K$ with an isometry. We now do this.

Let $\hat G_K^{rB}$ denote $G_K^{rB}$
with the metric induced by $h$ and let $\hat G_0^{rB}$ denote
$G_0^{rB}$ with the pullback metric via $G_0^{rB}\to
G_K^{rB}\overset{h}{\to} G_K^{B}$. Edgelets and colors of $\hat
G_0^{rB}$ and $\hat G_K^{rB}$ are induced by the natural
homeomorphisms with respectively $G_0^{rB}$ and $G_K^{rB}$. We
identify $r$ with its image in $\hat G_K^{rB}$.  By construction, we
have morphisms $\hat G_0^{rB}\to \hat G^{rB}_K = G_K^{B}\to \Gamma$.
It remains to check that the composition is optimal.

By construction, our composition is injective when restricted to
edgelets. Since $G_0^{B}\to \Gamma$ is optimal, a vertex of $\hat
G_0^{rB}$ with no incident edgelets in the orbit of the preimage of
$r$ has two gates. Since $r$ is interior to a natural edge
of $G^{rB}_K$, turns at endpoints of edgelets in $r$ are legal
with respect to $\hat G_K^{rB}\to\Gamma$. Hence, we only need to show
that an endpoint of an edgelet of $\hat G_0^{rB}$ in the orbit of the
preimage of $b$ has two gates with respect to $\hat G_0^{rB}\to
\hat G_K^{rB}$. But this is clear since $G_0^{rB}\to G_K^{rB}$ is
optimal.
\end{proof}

\begin{cor}\label{hm4}
There is a constant $M=M(rank~\FF_n)$ so that the following holds.
Suppose we are given a combing diagram which is a vertical stack of
$N>1$ combing rectangles with the projection to $\S'$ of the bottom row
of diameter $>MN$. 

$$
\begindc{2}[30]
\obj(1,1)[11]{$A$}[\south]
\obj(5,1)[51]{$C$}[\south]
\obj(1,5)[15]{$A'$}
\obj(5,5)[55]{$C'$}
\obj(1,2)[12]{}
\obj(5,2)[52]{}
\obj(1,4)[14]{}
\obj(1,5)[15]{}
\obj(5,4)[54]{}
\obj(1,5)[15]{}
\mor{11}{51}{}[\atright,\solidarrow]
\mor{12}{52}{}[\atright,\solidarrow]
\mor{14}{54}{}[\atright,\solidarrow]
\mor{15}{55}{}[\atright,\solidarrow]
\mor{11}{12}{}
\mor{51}{52}{}
\mor{12}{14}{}[\atright,\dotline]
\mor{14}{15}{}[\atright,\solidline]
\mor{52}{54}{}[\atright,\dotline]
\mor{54}{55}{}[\atright,\solidline]
\enddc
$$

Then there is a combing diagram consisting of a
stack of two rectangles with top and bottom paths extended to folding
paths so that:
\begin{itemize}
\item the bottom folding path equals the bottom folding path of the
  given diagram,
\item the top right corner equals the top right corner of the given
  diagram,
\item if the original top folding path is extended to the right, the
  same extension forms a folding path in the resulting diagram,
\item the vertical arrows in the two resulting rectangles are down-up,
\item paths not included in the rectangles have diameter $<MN$ in
  $\S'$.
\end{itemize}

$$
\begindc{1}[5]
\obj(6,6)[11]{$A$}[\south]
\obj(30,6)[51]{$C$}[\south]
\obj(30,30)[55]{$C'$}
\obj(6,12)[12]{}
\obj(6,18)[13]{$A''$}
\obj(18,6)[31]{$B$}[\south]
\obj(18,12)[32]{}
\obj(18,18)[33]{$B''$}
\cmor((18,18)(20,18)(21,19)(24,24)(27,29)(28,30)(30,30)) \pright(30,30){}
\mor{11}{51}{}[\atright,\solidarrow]
\mor{12}{32}{}[\atright,\solidarrow]
\mor{11}{12}{}
\mor{13}{12}{}
\mor{12}{32}{}
\mor{13}{33}{}
\mor{33}{32}{}
\mor{31}{32}{}
\enddc
$$
\end{cor}

\begin{proof}
When $N=2$ this is Theorem~\ref{hm}. Proceed by induction, combining
consecutive collapse maps into a single collapse map.
\end{proof}

\subsection{Verification of Masur-Minsky axioms}

To prove that $\S'$ is hyperbolic and that folding paths are
reparametrized geodesics it suffices to verify:
\begin{enumerate}[(1)]
\item Folding paths form a coarsely transitive family of paths.
\item For each folding path $G_t$ there is a coarse Lipschitz
  retraction $\S'\to G_t$.
\item Every folding path is uniformly strongly contracting.
\end{enumerate}
See \cite{MM} and also \cite[Section~3]{S}.

We now proceed to verify these properties.

\subsubsection{Folding paths are coarsely transitive}

If $G_1,G_2\in\S'$ let $G_i'\in\S'$ be a rose obtained from $G_i$ by
expanding, then collapsing. After adjusting the metric on $G_1'$, there
is a folding path from $G_1'$ to $G_2'$. 

\subsubsection{Definition of the projection.}

Let $G_t$ be a folding path and $\Gamma\in\S'$. Following \cite{S},
define the {\it projection} $\pi(\Gamma)$ of $\Gamma$ to the folding
path $G_t$ as the time $T$ (or the graph $G_T$), where $T$ is the sup
of the times $t$ such that there is a combing diagram as pictured where $\Gamma$ denotes any representative of $\Gamma$.

$$
\begindc{1}[5]
\obj(6,6)[11]{$G_0$}[\south]
\obj(30,6)[51]{$G_L$}[\south]
\obj(30,30)[55]{$\Gamma$}
\obj(6,12)[12]{}
\obj(6,18)[13]{}
\obj(18,6)[31]{$G_t$}[\south]
\obj(18,12)[32]{}
\obj(18,18)[33]{}
\cmor((18,18)(20,18)(21,19)(24,24)(27,29)(28,30)(30,30)) \pright(30,30){}
\mor{11}{51}{}[\atright,\solidarrow]
\mor{12}{32}{}[\atright,\solidarrow]
\mor{11}{12}{}
\mor{13}{12}{}
\mor{12}{32}{}
\mor{13}{33}{}
\mor{33}{32}{}
\mor{31}{32}{}
\enddc
$$
If no such combing diagram exists, the sup is defined to be 0. (It can
be shown that it always exists, i.e. there is a collapse followed by
an expansion $G_0\to R\leftarrow G_0'$ and an optimal map $G_0'\to
\Gamma$.) 

\subsubsection{The projection is a coarse retraction}

Let $\Gamma=G_{t_0}$ (as underlying graphs, but with possibly
different metrics). Then clearly $T=\pi(\Gamma)\geq t_0$, since we can
construct combing rectangles from $G_{t_0}\to R\leftarrow \Gamma$, where
$R$ has one edge. Suppose we have a combing diagram as pictured.

$$
\begindc{1}[5]
\obj(6,6)[11]{$G_0$}[\south]
\obj(12,6)[21]{$G_{t_0}$}[\south]
\obj(12,18)[23]{$G'_{t_0}$}[\north]
\obj(12,12)[22]{}
\obj(30,6)[51]{$G_L$}[\south]
\obj(30,30)[55]{$\Gamma$}
\obj(6,12)[12]{}
\obj(6,18)[13]{}
\obj(18,6)[31]{$G_t$}[\south]
\obj(18,12)[32]{}
\obj(18,18)[33]{}
\cmor((18,18)(20,18)(21,19)(24,24)(27,29)(28,30)(30,30)) \pright(30,30){}
\mor{21}{22}{}
\mor{23}{22}{}
\mor{11}{51}{}[\atright,\solidarrow]
\mor{12}{32}{}[\atright,\solidarrow]
\mor{11}{12}{}
\mor{13}{12}{}
\mor{12}{32}{}
\mor{13}{33}{}
\mor{33}{32}{}
\mor{31}{32}{}
\enddc
$$

We must show that $d_{\S'}(G_{t_0},G_{t})$ is uniformly
bounded. Apply Lemma~\ref{technical2} to the
optimal map $G'_{t_0}\to \Gamma=G_{t_0}$ (since $G_{t_0}$ and
$G'_{t_0}$ collapse to the same graph, they each have an edge that
determines the same splitting). The conclusion is that the folding
path from $G'_{t_0}$ to $G'_t$ has bounded image in $\S'$, so the same
is true for the path $[G_{t_0},G_t]$.

\subsubsection{The projection is coarsely Lipschitz}

Suppose $\Gamma\to\Gamma'$ is a collapse. We first observe that
$\pi(\Gamma)\geq\pi(\Gamma')$. The argument is in the diagram below.
Construct a combing rectangle using the folding path represented by
the horizontal line ending in $\Gamma'$ and the collapse map
$\Gamma\to\Gamma'$ (note that this may require inserting constant
paths and reparametrizing) and combine it with the given combing
rectangles. The top two rectangles on the left form one combing
rectangle.

$$
\begindc{1}[5]
\obj(6,6)[11]{$G_0$}[\south]
\obj(30,6)[51]{$G_L$}[\south]
\obj(30,18)[53]{$\Gamma'$}[\south]
\obj(18,24)[34]{}
\obj(6,24)[14]{}
\obj(30,24)[54]{$\Gamma$}[\north]
\obj(6,12)[12]{}
\obj(6,18)[13]{}
\obj(18,6)[31]{$G_t$}[\south]
\obj(18,12)[32]{}
\obj(18,18)[33]{}
\mor{14}{34}{}
\mor{34}{54}{}
\mor{14}{13}{}
\mor{34}{33}{}
\mor{54}{53}{}
\mor{11}{51}{}[\atright,\solidarrow]
\mor{12}{32}{}[\atright,\solidarrow]
\mor{11}{12}{}
\mor{13}{12}{}
\mor{12}{32}{}
\mor{13}{33}{}
\mor{33}{32}{}
\mor{31}{32}{}
\mor{33}{53}{}
\enddc
$$

Now we must show that the segment $[\pi(\Gamma'),\pi(\Gamma)]$ has a
uniformly bounded projection in $\S'$. This follows from Corollary~\ref{hm4} with $N=3$ applied to the stack of 3 rectangles in the
diagram below.

$$
\begindc{1}[5]
\obj(6,6)[11]{$G_0$}[\south]
\obj(30,6)[51]{$G_L$}[\south]
\obj(30,18)[53]{$\Gamma$}[\south]
\obj(18,24)[34]{}
\obj(6,24)[14]{}
\obj(30,24)[54]{$\Gamma'$}[\north]
\obj(6,12)[12]{}
\obj(6,18)[13]{}
\obj(18,6)[31]{$G_t$}[\south]
\obj(18,12)[32]{}
\obj(18,18)[33]{}
\mor{14}{34}{}
\mor{34}{54}{}
\mor{13}{14}{}
\mor{33}{34}{}
\mor{53}{54}{}
\mor{11}{51}{}[\atright,\solidarrow]
\mor{12}{32}{}[\atright,\solidarrow]
\mor{11}{12}{}
\mor{13}{12}{}
\mor{12}{32}{}
\mor{13}{33}{}
\mor{33}{32}{}
\mor{31}{32}{}
\mor{33}{53}{}
\enddc
$$

\subsubsection{Folding paths are uniformly strongly contracting}

Let $G_t$ be a folding path and $\Gamma\in\S'$ so that
$d_{\S'}(\Gamma,G_t)>R$ for every $t$ and $R$ is large. Suppose
$\Gamma'\in\S'$ is such that $d_{\S'}(\Gamma,\Gamma')<\mu R$ for a certain
$\mu>0$. After possibly adjusting the metric on $\Gamma'$ there is a
zig-zag path from $\Gamma$ to $\Gamma'$ of length $<\mu R$. Take a
combing diagram associated with the projection of $\Gamma$ to $G_t$
and add to it the stack of combing rectangles induced by the zig-zag
path (in the diagram $d_{\S'}(\Gamma,\Gamma')=3$).

$$
\begindc{2}[30]
\obj(1,1)[11]{$G_0$}[\south]
\obj(5,1)[51]{$G_L$}[\south]
\obj(3,1)[31]{$G_t$}[\south]
\obj(1,5)[15]{}
\obj(5,5)[55]{}
\obj(1,2)[12]{}
\obj(3,2)[32]{}
\obj(1,3)[13]{}
\obj(3,3)[33]{}
\obj(5,3)[53]{$\Gamma$}[\south]

\obj(1,4)[14]{}

\obj(5,4)[54]{}

\obj(1,5)[15]{}

\obj(5,5)[55]{}

\obj(1,6)[16]{}

\obj(5,6)[56]{$\Gamma'$}

\mor{11}{51}{}[\atright,\solidarrow]
\mor{12}{32}{}[\atright,\solidarrow]
\mor{14}{54}{}[\atright,\solidarrow]
\mor{15}{55}{}[\atright,\solidarrow]
\mor{16}{56}{}[\atright,\solidarrow]
\mor{11}{12}{}[\atright,\solidarrow]
\mor{13}{12}{}[\atright,\solidarrow]
\mor{31}{32}{}[\atright,\solidarrow]
\mor{33}{32}{}[\atright,\solidarrow]
\mor{13}{53}{}[\atright,\solidarrow]
\mor{13}{16}{}[\atright,\solidline]
\mor{53}{56}{}[\atright,\solidline]
\enddc
$$

Apply Corollary~\ref{hm4} (with $N$ equal to the height of the
stack) to replace the stack of rectangles
with a stack of height two, plus the trailing folding paths. The
assumption on the distances guarantees that this stack continues beyond
the combing rectangles for the projection.

$$
\begindc{2}[30]
\obj(1,1)[11]{$G_0$}[\south]
\obj(5,1)[51]{$G_L$}[\south]
\obj(3,1)[31]{$G_t$}[\south]
\obj(1,5)[15]{}
\obj(4,3)[43]{}
\obj(4,4)[44]{}
\obj(4,5)[45]{}

\obj(1,2)[12]{}
\obj(3,2)[32]{}
\obj(1,3)[13]{}
\obj(3,3)[33]{}
\obj(5,3)[53]{$\Gamma$}[\south]
\obj(1,4)[14]{}
\obj(1,5)[15]{}
\obj(5,6)[56]{$\Gamma'$}

\mor{11}{51}{}[\atright,\solidarrow]
\mor{12}{32}{}[\atright,\solidarrow]
\mor{14}{44}{}[\atright,\solidarrow]
\mor{15}{45}{}[\atright,\solidarrow]

\mor{11}{12}{}[\atright,\solidarrow]
\mor{13}{12}{}[\atright,\solidarrow]
\mor{31}{32}{}[\atright,\solidarrow]
\mor{33}{32}{}[\atright,\solidarrow]
\mor{13}{53}{}[\atright,\solidarrow]
\mor{13}{15}{}[\atright,\solidline]
\mor{43}{45}{}[\atright,\solidline]
\cmor((4,5)(5,5)(5,6)) \pup(6,6){}
\enddc
$$

Apply Corollary~\ref{hm4}
one more time with $N=4$ and deduce that the projection of $\Gamma'$
cannot precede the projection of $\Gamma$ by more than a bounded
amount. Reversing the roles of $\Gamma$ and $\Gamma'$ proves the
statement. 

To summarize, we have

\begin{thm}
$\S'$ (and hence also $\S$) is hyperbolic and liberal folding paths are uniform
  reparametrized quasi-geodesics.
\end{thm}

\bibliography{./ref}

\end{document}